\documentclass{amsart}

\usepackage{amsthm} 
\usepackage{color}

\usepackage{todonotes}

\usepackage{graphicx,epsfig}
\usepackage{epstopdf}
\usepackage{amsmath, amssymb, latexsym, euscript,amscd}
\usepackage{url}
\usepackage[all]{xy}
\usepackage{psfrag}
\usepackage[pagebackref=true]{hyperref}
\usepackage{mathrsfs}
\usepackage{tikz}
\usetikzlibrary{shapes.geometric}
\usepackage{subcaption}
\usepackage{multicol}

\usepackage{fullpage}

\usetikzlibrary{shapes.geometric}
\usetikzlibrary{positioning} 
\usepackage[alphabetic]{amsrefs}

\usepackage{thmtools} 
\usepackage{thm-restate}

\newcounter{notes}%

\definecolor{darkgreen}{rgb}{0.0, 0.5, 0.0}

\newtheorem{theorem}{Theorem}[section]
\newtheorem{lemma}[theorem]{Lemma}
\newtheorem{corollary}[theorem]{Corollary} 
\newtheorem{definition}[theorem]{Definition} 
\newtheorem{proposition}[theorem]{Proposition}
\newtheorem{remark}[theorem]{Remark}

\def\smallskip{\vspace{.15cm}}
\def\medskip{\vspace{.3cm}}
\def\text{\mbox}
\def\rh2{{\mathbb R}{\mathbb H}^2}
\def\ch2{{\mathbb C}{\mathbb H}^2}
\def\RR{{\mathbb R}}
\def\CC{{\mathbb C}}
\def\QQ{{\mathbb Q}}
\def\HH{{\mathbb H}}

\def\ZZ{{\mathbb Z}}
\def\NN{{\mathbb N}}
\def\FF{{\mathbb F}}
\def\HH{{\mathbb H}}

\def\RP2{{\mathbb{RP}}^2}
\def\RP3{{\mathbb{RP}}^3}
\def\RP{{\mathbb{RP}}}

\def\Id{\operatorname{Id}}

\def\SL{\operatorname{SL}}
\def\Id{\operatorname{Id}}

\def\PSL{\operatorname{PSL}}
\def\SO{\operatorname{SO}}

\def\PGL{\operatorname{PGL}}
\def\GL{\operatorname{GL}}
\def\int{\operatorname{int}}
\def\Inn{\operatorname{Inn}}

\def\Stab{\operatorname{Stab}}

\def\H2R{{\mathbb H}^2\times {\mathbb R}}

\def\C2{\operatorname{C^2}}

\def\Aut{\mathrm{Aut}}

\def\Stab{\mathrm{Stab}}
\def\Fix{\mathrm{Fix}}
\def\abstrc{\edit{\mathrm{abstrc}}}
\def\Fix{\edit{\mathrm{Fix}}}
\def\Diag{\edit{\mathrm{Diag}}}
\def\det{\edit{\mathrm{det}}}

\newcommand{\edit}[1]{{\color{black}#1}} 
\newcommand{\editA}[1]{{\color{black}#1}} 
\newcommand{\editB}[1]{{\color{black}#1}}

\definecolor{back}{RGB}{255,255,255}
\definecolor{fore}{RGB}{0,0,0}
\definecolor{title}{RGB}{255,0,90}

\definecolor{green}{rgb}{0.0, 0.5, 0.0}
\definecolor{purple}{rgb}{0.5, 0.0, 0.5}
\definecolor{bluegreen}{rgb}{0.0,0.5, 0.5}
\definecolor{orange}{rgb}{1,0.5, 0.1}
\definecolor{redgreen}{rgb}{0.5, 0.5, 0.0}

\def\green{\color{green}}

\def\green{\color{green}}

\def\g2{{\green 2}}

\newcommand{\bv}{\left[\begin{array}{c}}
\newcommand{\ev}{\end{array}\right]}
\newcommand{\bbmat}{\begin{bmatrix}} 
\newcommand{\ebmat}{\end{bmatrix}}
\newcommand{\bmat}{\begin{matrix}}
\newcommand{\emat}{\end{matrix}}
\newcommand{\bpmat}{\begin{pmatrix}} 
\newcommand{\epmat}{\end{pmatrix}}

\begin{document}
\title{Chabauty Limits of Groups of Involutions In $\SL(2,F)$ for local fields} 
\author{Corina Ciobotaru}
\address{Aarhus Institute of Advanced Studies and Department of Mathematics, Aarhus University, Ny Munkegade 118, 8000 Aarhus C, Denmark}
\email{corina.ciobotaru@gmail.com}

\author{Arielle Leitner}
\address{Afeka College of Engineering, 38 Mivtsa Hakadesh st, Tel Aviv, Israel}
\email{ariellel@afeka.ac.il}
%\thanks{corina.ciobotaru@gmail.com}\thanks{aleitner.math@gmail.com}
\date{September 23, 2023}

\begin{abstract}
We classify Chabauty limits of groups fixed by various (abstract) involutions over $\SL(2,F)$, where $F$ is a finite field-extension of $\QQ_p$, with $p\neq 2$.   To do so, we first classify abstract involutions over $\SL(2,F)$ with $F$ a quadratic extension of $\QQ_p$, and prove $p$-adic polar decompositions with respect to various subgroups of $p$-adic $\SL_2$.  Then we classify Chabauty limits of: $\SL(2, F) \subset \SL(2,E)$ where $E$ is a quadratic extension of $F$, of $\SL(2, \RR) \subset \SL(2, \CC)$, and of $H_\theta \subset \SL(2,F)$, where $H_\theta$ is the fixed point group of an $F$-involution $\theta$ over $\SL(2,F)$.
\end{abstract}

\maketitle

\section{Introduction} %%%%%%%%%%%%%%%%%%%

Let $k$ be a local field, which is not necessarily algebraically closed. Let $G$ be a reductive algebraic group defined over $k$. A \textit{symmetric $k$-variety} of $G$ is the quotient $G(k)/H(k)$, where $H$ is the fixed point group of an involution $\theta$ defined over $k$ of $G$, and $G(k), H(k)$ denote the sets of $k$-rational points of $G$, $H$. Symmetric $k$-varieties appear naturally and play a central role in the representation theory of algebraic groups, the Langlands program, or the Plancherel formulas for Riemannian symmetric spaces (see \cite[Introduction]{Hel_k_invol}).

When $k=\RR$, or $\CC$, a space $G(k)/H(k)$ as above is an \textit{affine symmetric space}, generalizing the theory of Riemannian symmetric spaces. Familiar examples of affine symmetric spaces come from quadratic forms on $\RR^n$, or $\CC^n$, of signature $(p, q)$, where the fixed point group $H(k)$ is the corresponding orthogonal group  $O(p,q)$ preserving that quadratic form. In this way one obtains  spherical geometry, hyperbolic geometry, de Sitter geometry, or anti de Sitter geometry. Although  these geometries have different curvature, one can find \textit{geometric transitions} between them: continuous paths of geometric structures that change the type of the model geometry in the limit, also known as a \textit{limit of geometries}.  Geometric transitions arise  in physics: 
deforming general relativity into special relativity, or quantum mechanics into Newtonian mechanics. Cooper, Danciger and Weinhard \cite{CDW} classify limits of geometries coming from affine symmetric spaces over $\RR$.   In particular, they classify the limits of geometries of all of the groups $\SO(p,q)$ inside of $\GL(n, \RR)$.  Their approach uses a root space decomposition  of real Lie algebras.  Trettel \cite{Trettel} provides another approach in Chapter 6 of his PhD thesis using the wonderful compactification. 

 Felix Klein’s Erlangen Program encodes geometries and uniquely determines them by their groups of isometries. Therefore, studying limits of geometries is equivalent to studying limits of groups of isometries (i.e. Lie groups) in the \emph{Chabauty topology}, (see Section \ref{sec::Chabauty} for the concrete definitions and properties).

Our article is the first installment of an analogous classification of Chabauty limits as in \cite{CDW} for $p$-adic groups fixed by (abstract) involutions over $\SL(2,k)$, where $k$ is a finite field-extension of $\QQ_p$ with $p\neq 2$. In future work we will study the general case of $\SL(n,k)$.  

We  use the classification of the  isomorphism classes of $k$-involutions of  a connected reductive algebraic group defined over $k$  given in \cite{Hel_k_invol}. A simple characterization of the isomorphism classes of $k$-involutions of $\SL(n,k)$ is given in \cite{HelWD}. Further, \cite{Beun, HW_class, Sutherland} study $\QQ_p$-involutions and their fixed point groups for $\SL(2,\QQ_p)$.

Using the results of Borel--Tits \cite{BoTi}, and Steinberg \cite{Stein}, our first main result is in Section \ref{sec:autom} and gives a classification of all abstract and $E$-involutions $\theta$ of $\SL(2,E)$, where $E= \QQ_p (\alpha)$ is a quadratic extension of $\QQ_p$. %We show in Theorem \ref{thm_main::inv_SL_E}  \edit{that} they are all the composition of an inner automorphism $\iota_A$ of $\GL(2,E)$ with the identity or the 
Let $\sigma$ be the field conjugation automorphism on $E$ (i.e.  $\sigma : E \to E$ given by $\sigma (a + \alpha b) = a - \alpha b$). 
% and we determine the matrices $A \in \GL(2,E) $ for the inner automorphism $\iota_A$.  
 \editA{If $A \in \SL(2,E)$ is a matrix and $\gamma \in \Aut(E)$ is a field automorphism, then $\gamma(A)$ is the matrix where $\gamma$ is applied to every matrix entry. }

\begin{theorem}[\edit{See Thm.  \ref{thm::inv_SL_E}} ]
\label{thm_main::inv_SL_E}
Let  $E=\QQ_p(\alpha)$ be a quadratic extension of $\QQ_p$. Any abstract involution $\theta$ of $\SL(2,E)$ is of the form $ \theta = \iota_A \circ \gamma$ \edit{where $\gamma \in \{ \Id, \sigma \}$ and the matrices $A$ are written explicitly.} 
\iffalse 
\edit{either}: 
\begin{enumerate}
\item
$\gamma= \sigma$  and $A \in \edit{ \left \{\left( \begin{smallmatrix} z &y \\ 1 & -\sigma(z) \end{smallmatrix} \right),  \left( \begin{smallmatrix} x &0 \\ 0 & 1 \end{smallmatrix} \right)\right \}} \subset \GL(2,E)$, with $y \in \QQ_p, z,x\in E$, with $z\sigma(z)+ y \neq 0$ and $x\sigma(x)=1,$
\item
\edit{or} $\gamma=\Id$ and \edit{$A$ is $\SL(2,E)$-conjugate to a matrix of the form} $\left( \begin{smallmatrix} 0 &1 \\ a &0 \end{smallmatrix} \right)$, with $a\in E^{*}/(E^{*})^{2}$. %up to conjugacy.
\end{enumerate}
\fi
\end{theorem}

We further compute the fixed point groups $H_\theta$ of the  involutions $\theta$ from Theorem \ref{thm_main::inv_SL_E} using the ends in the ideal boundary $\partial T_K$ of the Bruhat--Tits tree $T_K$ of $\SL(2,K)$, where $K$ is finite field-extension of $E$, and thus of $\QQ_p$, that is chosen accordingly. We will show those fixed point groups are either trivial, or compact, or  $\GL(2,E)$-conjugate to either the diagonal in $\SL(2, E)$, or to $\SL(2, \QQ_p)$.

 In Section \ref{sec:autom} we obtain a geometric interpretation  of the fixed point groups $H_{\theta_a}$ of the $k$-involutions computed in \cite{HW_class} for $\SL(2,k)$, when $k$ is a  finite field-extension of $\QQ_p$.  %(i.e. $\theta_a$ as in Theorem  \ref{thm_main::inv_SL_E}(2)); 
\edit{We compute the \editB{fixed point} groups of involutions for the two cases of Theorem \ref{thm_main::inv_SL_E}.}

%Next, we classify the rest of  $H_\theta$  given by Theorem \ref{thm_main::inv_SL_E}.
\begin{theorem}[\edit{See Cor. \ref{lem::geom_k_inv_SL_2}, Thm. \ref{thm::inv_SL_E_2}}]
\label{thm_main::inv_SL_E_2}
\begin{enumerate} 
\item \edit{Let $F$ be a finite field extension of $\QQ_p$, take $\gamma = \sigma$, and $a\in F^{*}/(F^{*})^{2}$, 
 then the \editB{fixed point} groups of involutions of $\theta_a$ are of the form:}
$$H_{\theta_a}: =\{  \left( \begin{smallmatrix} x &y \\ ay &x\end{smallmatrix} \right) \in \SL(2,F)\; \vert \; x^2 - ay^2 =1 \}. $$  
%For a precise and technical statement see Corollary \ref{lem::geom_k_inv_SL_2}.
%\edit{The element $a$ depends on a certain field extension $K_a = F (\sqrt a)$.}
\item Let  $E=\QQ_p(\alpha)$ be a quadratic extension of $\QQ_p$, and \edit{$\gamma = \Id$}.
%$\theta$ as in Theorem  \ref{thm_main::inv_SL_E}(1).
Then $H_{\theta}:=\{g \in \SL(2,E) \; \vert \; \theta(g)= g\}$ \edit{is a ``conjugate" of $\SL(2,\QQ_p)$, given explicitly.} 
\end{enumerate} 
\iffalse
\edit{where}, either:
\begin{enumerate}
\item
$H_{\theta}$ is  $\GL(2,E)$-conjugate to $\SL(2,\QQ_p)$
\item
$H_{\theta}=B \SL(2,K^{\sigma}) B^{-1} \cap \SL(2,E)$, for some matrix  $B \in \GL(2,K)$, where $K \neq E$ is some finite field-extension of $E$ and $K^{\sigma}:=\{x \in K \; \vert \; \sigma(x)=x\}$ is the maximal subfield of $K$ %with that property. 
\edit{fixed by $\sigma$}. In particular,  $H_\theta$ is compact or trivial.
\end{enumerate}
\fi 
\end{theorem} 

One of the strategies  to compute Chabauty limits of the fixed point groups $H$ of involutions over $G$ is to employ the polar decomposition $G=\mathcal{K}\mathcal{B}H$, where $\mathcal{K}$ is a compact subset of $G$, and $\mathcal{B}$ is a union of split tori of $G$. If moreover $\mathcal{B}$ is a finite union of such tori, then it is enough to compute the desired Chabauty limits  under conjugation with a sequence of elements from some fixed split  torus in $\mathcal{B}$. We cannot apply directly \cite{BenoistOh} as their result is proven only for $k$-involutions and not for abstract involutions. In Section \ref{sec:polar} we prove a polar decomposition for $k$-involutions and abstract involutions for the $p$-adic $\SL_2$. Along the way we provide different, direct, and more geometric proofs than in \cite{HW_class,HelWD,Hel_k_invol} for the case of  $\SL_2$.   
%Below $\theta_a$  refers to the involutions from Theorem \ref{thm_main::inv_SL_E}(2), and $\theta$ to the involutions from Theorem \ref{thm_main::inv_SL_E}(1).

\begin{proposition}[\edit{See Prop. \ref{prop::polar_decom}}, polar decomposition for various subgroups of $\SL_2$]
\label{prop_main::polar_decom}
Let $F$ be a finite field-extension of $\QQ_p$ and $E=F(\alpha)$ be a quadratic extension of $F$, \edit{and \editB{let} $\theta$, $\theta_a$ \editB{be} the involutions from Theorems \ref{thm_main::inv_SL_E}, \ref{thm_main::inv_SL_E_2}. } Let $H \leq G$ be \edit{one of the following pairs} 
$$(G,H) \in \{(\SL(2,F), H_{\theta_a}), (\SL(2,E),\SL(2,F)), (\SL(2,\QQ_p(\alpha)),H_{\theta})\}.$$
%Denote by $T_G$ the  Bruhat--Tits tree of $G$ and by $T_H$ the (possibly finite)  \edit{$H$-invariant} subtree of $T_G$. Let $I_H$ be the number of $H$-orbits in the ideal boundary $\partial T_G -\partial T_H$, and $\xi_i$ a representative in each such orbit. Let $x_0 \in T_G \cap T_H$ be a vertex, which for the pair $ (\SL(2,E),\SL(2,F))$ will be taken to be the point $0$ as in the \edit{Figures} \ref{fig::fig_unram}, \ref{fig::fig_ram}. 
Then \edit{there is a decomposition}
$$G= \mathcal{K}\mathcal{B}H$$ where $\mathcal{K}$ is a \edit{specific} compact subset of $G$, and $\mathcal{B}= \{\Id\} \bigsqcup\limits_{i\in I_H} A_i$, where $A_i:= \{a_i^{n} \; \vert \; n \in \ZZ\}$ with $a_i \in G$  a hyperbolic element of translation length $2$ and with attractive endpoint in the $H$-orbit of \edit{an end of the tree for $G$}. 
%$\xi_i$.
\end{proposition}

Finally, in Sections \ref{sec:chab_quad}, \ref{sec:chab_real}, \ref{sec:chab_sym} we enumerate  all Chabauty limits of various fixed point groups of involutions over $\SL(2,k)$. The subgroups $B_{k}^{+}$ \edit{are} the upper triangular Borel subgroups of $\SL(2,k)$.  %See the relevant sections for more precise statements. 

\begin{theorem}[\edit{See Thm. \ref{cor::cartan_lim_SL}}]
\label{thm_main::cartan_lim_SL}
Let $F$ be a finite field-extension of $\QQ_p$ and $E=F(\alpha)$ be a quadratic extension of $F$, \edit{so $\alpha^2 \in  F^{*}/(F^{*})^{2}, \alpha \neq 1$}. Then  any Chabauty limit of $\SL(2,F)$ inside  $\SL(2,E)$ is  $\SL(2,E)$-conjugate to either $\SL(2,F)$, or to the subgroup 
$\{\left( \begin{smallmatrix} a-\alpha b& z \\  0 &  a+\alpha b \end{smallmatrix} \right) \; \vert \;  a,b \in F \text{ with }a^2 - \alpha^2 b^{2}=1,  z \in E \} \leq B_E^{+}.$ 
\end{theorem}

\begin{theorem}[\edit{See Thm. \ref{thm::cartan_lim_SL_r}}]
\label{thm_main::cartan_lim_SL_r}
 Any Chabauty limit of $\SL(2,\RR)$ inside  $\SL(2,\CC)$ is  $\SL(2,\CC)$-conjugate to either $\SL(2,\RR)$, or to  the subgroup 
$\{\left( \begin{smallmatrix} a-i b& z \\  0 &  a+i b \end{smallmatrix} \right) \; \vert \;  a,b \in \RR \text{ with }a^2 + b^{2}=1,  z \in \CC\} \leq B_{\CC}^{+}.$ 
\end{theorem}

\begin{theorem}[\edit{See Thm. \ref{thm::k_inv_chabauty}}]
\label{thm_main::k_inv_chabauty}
Let $F$ be a finite field-extension of $\QQ_p$, and $H_{\theta_a}$ %the fixed point group of an involution $\theta_a$ on $\SL(2,F)$ as in Theorem \ref{thm_main::inv_SL_E} \edit{for $\gamma = \sigma$.}
\edit{as in Theorem \ref{thm_main::inv_SL_E_2}(1)}.
Then any Chabauty limit of $H_{\theta_a}$ is either $\SL(2,F)$-conjugate to $H_{\theta_a}$, or to the subgroup \edit{ $\{ \mu \left( \begin{smallmatrix} 1& x \\ 0 &1  \end{smallmatrix} \right)\; |\; x \in F, \mu \in \mu_2  \}$} of the Borel  $B_{F}^{+}\leq \SL(2,F)$, where $\mu _2$ is the group of $2^{nd}$ roots of unity in  $F$.
\end{theorem}

%\emph{Outline:} The structure of the paper is as follows: Section \ref{sec:background} gives background material on involutions and symmetric subgroups.  In Section \ref{sec:autom} we classify all abstract and $E$-involutions on $\SL(2, E)$, where $E$ is a quadratic extension of $\QQ_p$.  We show they are given by composition with an inner involution and a field automorphism.  Then in Section \ref{sec:polar} we prove a polar decomposition for symmetric subgroups which will be useful in finding their Chabauty limits.  Section \ref{sec:chab_quad} computes Chabauty limits of $\SL(2, \QQ_p) \subset \SL (2, E)$, section \ref{sec:chab_sym} for general symmetric subgroups inside $\SL(2,\QQ_p)$ and Section \ref{sec:chab_real} for $\SL(2, \RR) \subset \SL(2, \CC)$. 

\subsection*{Acknowledgements} Ciobotaru was partially supported by the Institute of Mathematics of the Romanian Academy (IMAR), Bucharest, and The Mathematisches Forschungsinstitut Oberwolfach (MFO, Oberwolfach Research Institute for Mathematics). She would like to thank those two institutions for the perfect working conditions they provide. As well, Ciobotaru is supported by the European Union’s Horizon 2020 research and innovation program under the Marie Sklodowska-Curie grant agreement No 754513, The Aarhus University Research Foundation, and a research grant (VIL53023) from VILLUM FONDEN.  Leitner was supported by Afeka college of engineering.  We thank Uri Bader, Linus Kramer,  Thierry Stulemeijer, Maneesh Thakur, and Alain Valette for helpful discussions.  \editA{We also thank the referee for an incredibly thorough job of making many suggestions which greatly improved the clarity and readability of the paper.}

\section{ The Chabauty Topology}%%%%%%%%%%%%%%%%%%%%%%%%%%%%%
\label{sec::Chabauty}

%Thus, to compute limits of groups one needs a topology on the corresponding set of groups with respect to which such limits can make sense. More precisely, 
The \textit{Chabauty topology}  was introduced in 1950 by Claude Chabauty \cite{Ch} on the set of all closed subgroups of a locally compact group.  The initial motivation of Chabauty was to show that some sets of lattices of some locally compact groups are relatively compact, and to generalize a criterion of Mahler about lattices of $\RR^n$.  In the Chabauty topology  the set of all closed subgroups of a locally compact group is compact.  This implies that any sequence of closed subgroups admits a  convergent subsequence, and so it has at least one limit, called  a \textit{Chabauty limit}.  Therefore, limits of groups, and thus limits of geometries, are not empty notions, the real difficulty  is not proving their existence, but computing concretely which geometric types may be obtained as limits.   For a good introduction to Chabauty topology~\cite{Ch}  see~\cite{CoPau, Harpe, GJT} or~\cite[Section 2]{Haettel} and the references therein. We briefly recall some facts that are used in this paper.

\edit{For a locally compact topological space $X$, the set of all closed subsets of $X$ is denote by $\mathcal{F}(X)$. This is endowed with the Chabauty topology where every open set is a union of finite intersections of subsets of the form $O_K:=\{ F \in \mathcal{F}(X) \; \vert \; F \cap K =\emptyset\}$, where $K$ is a compact  subset of $X$, or $O'_U:=\{ F \in \mathcal{F}(X) \; \vert \; F \cap U \neq \emptyset\}$, where $U$ is an open subset of $X$. By \cite[Proposition~1.7, p.~58]{CoPau} the space $\mathcal{F}(X)$ is compact with respect to the Chabauty topology. Moreover, if $X$ is Hausdorff and second countable then $\mathcal{F}(X)$ is separable and metrizable, thus Hausdorff (see \cite[Proposition I.3.1.2]{CEM}).  Given a family $\mathcal{F}$ of closed subsets of $X$, it is natural to study the closure of $\mathcal{F}$ with respect to the Chabauty topology, $\overline{\mathcal{F}}$, and determine whether or not elements in the boundary  $\overline {\mathcal{F}} - \mathcal{F}$ satisfy the same properties as those in $\mathcal{F}$.  We call elements of $\overline{\mathcal{F}}$ the \textbf{Chabauty limits} of $\mathcal{F}$. The next proposition provides an equivalent (and easier) definition for the Chabauty topology on $\mathcal{F}(X)$ when $X$ is a locally compact metric space. }

\begin{proposition}(\cite[Proposition~1.8, p.~60]{CoPau}, \cite[Proposition I.3.1.3]{CEM})
\label{prop::chabauty_conv}
 Suppose $X$ is a locally compact metric space.
A sequence of closed subsets $\{F_n\}_{n \in \NN} \subset \mathcal{F}(X)$  converges to $F \in \mathcal{F}(X)$, with respect to the Chabauty topology on $\mathcal{F}(X)$, if and only if the following two conditions are satisfied:
\begin{itemize} 
\item[1)] For every $f \in F$ there is a sequence $\{f_n \in F_n\}_{n \in \NN}$ converging to $f$ \editA{with respect to the topology on $X$};
\item[2)] For every sequence $\{f_n \in F_n\}_{n \in \NN}$, if there is a strictly increasing subsequence $\{n_k\}_{k \in \NN}$ such that $\{f_{n_k} \in F_{n_k}\}_{k \in \NN}$ converges to $f$ \editA{with respect to the topology on $X$}, then $f \in F$.
\end{itemize}
\end{proposition}

For a locally compact group $G$ we denote by $\mathcal{S}(G)$ the set of all closed subgroups of $G$. By \cite[Proposition~1.7, p.~58]{CoPau} the space $\mathcal{S}(G)$ is closed in $\mathcal{F}(G)$, with respect to the Chabauty topology, and is compact. Moreover, Proposition~\ref{prop::chabauty_conv}  applied to a sequence of closed subgroups $\{H_n\}_{n \in \NN} \subset \mathcal{S}(G)$ converging to $H \in \mathcal{S}(G)$,  yields a similar characterization of convergence in $\mathcal{S}(G)$.

Understanding the topology of the entire Chabauty space of closed subgroups of a group is difficult, and is known only for very few cases.  For example, it is easy to see \edit{that} $\mathcal{S}(\RR) \cong [ 0, \infty]$.  Hubbard and Pourezza \cite{HP} show $\mathcal{S} (\RR^2) \cong \mathbb{S}^4$, the 4-dimensional sphere, and Kloeckner \cite{Kloeckner} shows that while $\mathcal{S}(\RR^n)$ is not a manifold for $n >2$, it is a stratified space in the sense of Goresky--MacPherson, and is simply connected.  However a full description of $\mathcal{S}(\RR^n)$ is yet to be obtained. There are a few non-abelian
groups G for which $\mathcal{S}(G)$ is reasonably well understood, e.g. the Heisenberg group and some other low
dimensional examples \cite{BHK, Htt_2}, but for most $G$ the topology of $\mathcal{S}(G)$ is quite complicated.

Various authors have made progress understanding the closure of certain families of subgroups in $\mathcal{S}(G)$: abelian subgroups \cite{Baik1, Baik2, Haettel, Leitner_sl3, Leitner_sln}, connected subgroups \cite{LL}, and lattices \cite{BLL,Wang}.

In more recent years, $p$-adic Chabauty spaces have received attention. Bourquin and Valette \cite{BV} have described the homeomorphism type of $\mathcal{S}(\QQ_p ^*)$.   Cornulier \cite{Cornulier} has characterized several properties of $\mathcal{S}(G)$ for $G$ a locally compact abelian group. Chabauty closures of certain families of groups acting on trees have been studied by \cite{CR, Stulemeijer}, and there are several open questions about the Chabauty topology for locally compact groups in \cite{CM}.   The authors have studied limits of families of subgroups in $\SL(n, \QQ_p)$: parahoric subgroups \cite{CL} and Cartan subgroups \cite{CLV}.   Finally, \cite{GR} have studied compactifications of Bruhat-Tits buildings.

\edit{This article is the first stage in understanding a part of the $p$-adic Chabauty space for $\SL(n, \QQ_p)$ (a second article for $n \geq 3$ is forthcoming). We prove a $p$-adic analog of limits of groups preserving involutions, like \cite{CDW} do over $\RR$. }

\section{Background Material}\label{sec:background} %%%%%%%%%%%%%%%%%%%%%

Throughout this article we restrict to $p\neq 2$. Let $F$ be a finite field-extension of $\QQ_p$ and $E$ be any quadratic extension of $F$. Let $k_F, k_E$ be the residue fields of $F,E$, respectively,  and $\omega_F, \omega_E$ be  uniformizers of $F,E$, respectively. Recall $k_F^{*}/(k_F^{*})^{2} =\{1, S\}$, for some non-square $S \in k_F^{*}$. Then $F^{*}/(F^{*})^{2} =\{1, \omega_F, S, S\omega_F\}$ (\cite[Corollaries to Theorems 3 and 4]{Serre} or \cite[page 41, Section 12]{Sally}). \edit{We say} $E$ is \emph{ramified} then $E = F(\sqrt{\omega_F})$, or $E= F(\sqrt{S \omega_F})$ (where $\omega_F \neq \omega_E$), \edit{and we say} $E$ is \emph{unramified} \edit{if} $E=F(\sqrt{S})$ (where $\omega_F = \omega_E$).    We choose the unique valuation $\vert \cdot \vert_E$ on $E$  that extends the given valuation $\vert \cdot \vert_F$ on $F$. Choose $\alpha \in \{\sqrt{\omega_F}, \sqrt{S}, \sqrt{S\omega_F}\}$ and so $E=F(\alpha)$. Notice each element $x \in E$ can be uniquely written as $x = a+b \alpha$, with $a,b \in F$. For the ramified extensions we can consider $\omega_E^2 = \omega_F$. Let $\mathcal{O}_F:=\{ x \in F \; \vert \; \vert x \vert_F \leq 1 \}$ denote the ring of integers of $F$,  then $\mathcal{O}_F$ is compact and open in $F$. \editA{Moreover, one can choose $\omega_F \in \mathcal{O}_F, \omega_E \in \mathcal{O}_E$.} For $F=\QQ_p$ we have $\mathcal{O}_{\QQ_p}=\ZZ_p$, $\omega_{\QQ_p}=p$, $k_{\QQ_p}=\FF_p$, $\FF_p^{*}/(\FF_p^{*})^{2} =\{1, S_p\}$.

We denote by $T_F$ the Bruhat--Tits tree for $\SL(2,F)$ whose vertices  are equivalence classes of $\mathcal{O}_F$-lattices in $F^2$ (for its construction see \cite{Serre_tree}). The tree $T_F$ is a regular, infinite tree with valence $\vert k_F\vert +1$ at every vertex. The boundary at infinity $\partial T_{F}$ of $T_{F}$ is the projective space $P^1(F) \cong F \cup \{\infty\}$. Moreover, the endpoint $\infty \in \partial T_{F}$ corresponds to the vector  $\big[ \begin{smallmatrix}  0 \\1 \end{smallmatrix} \big] \in P^1(F)$. The rest of the endpoints $\xi \in \partial T_F$ correspond to the vectors $\big[ \begin{smallmatrix} 1 \\x \end{smallmatrix} \big] \in P^1(F) $, where $x \in F$. 

To give a concrete example,  the Bruhat--Tits tree of $\SL(2, \QQ_p)$ is the $p+1=\vert \FF_p \vert+1$-regular tree. The boundary at infinity $\partial T_{\QQ_p}$ of $T_{\QQ_p}$ is the projective space $P^1(\QQ_p)=\QQ_p \cup \{\infty\}$. In the figures below we give a concrete visualization of the Bruhat--Tits tree of $\SL(2, F)$ inside the  $\SL(2, E)$ \edit{tree}, when $F=\QQ_p$.
In both pictures we have drawn the Bruhat--Tits tree for $\SL(2, \QQ_p)$ (red) inside  the tree for $\SL(2,E)$ (blue and red).  We denote \edit{by} $x \in \QQ_p$, $y \in \QQ_p ^*$.

\editA{\begin{remark}  We denote quadratic field extensions in two ways: $F(\alpha)$ when we want to denote an arbitrary quadratic extension, or $F(\sqrt{a})$ when we wish to take $a$ to be a specific element for our computations. 
\end{remark}}

\begin{figure}
\includegraphics[width=12cm, height=6cm]{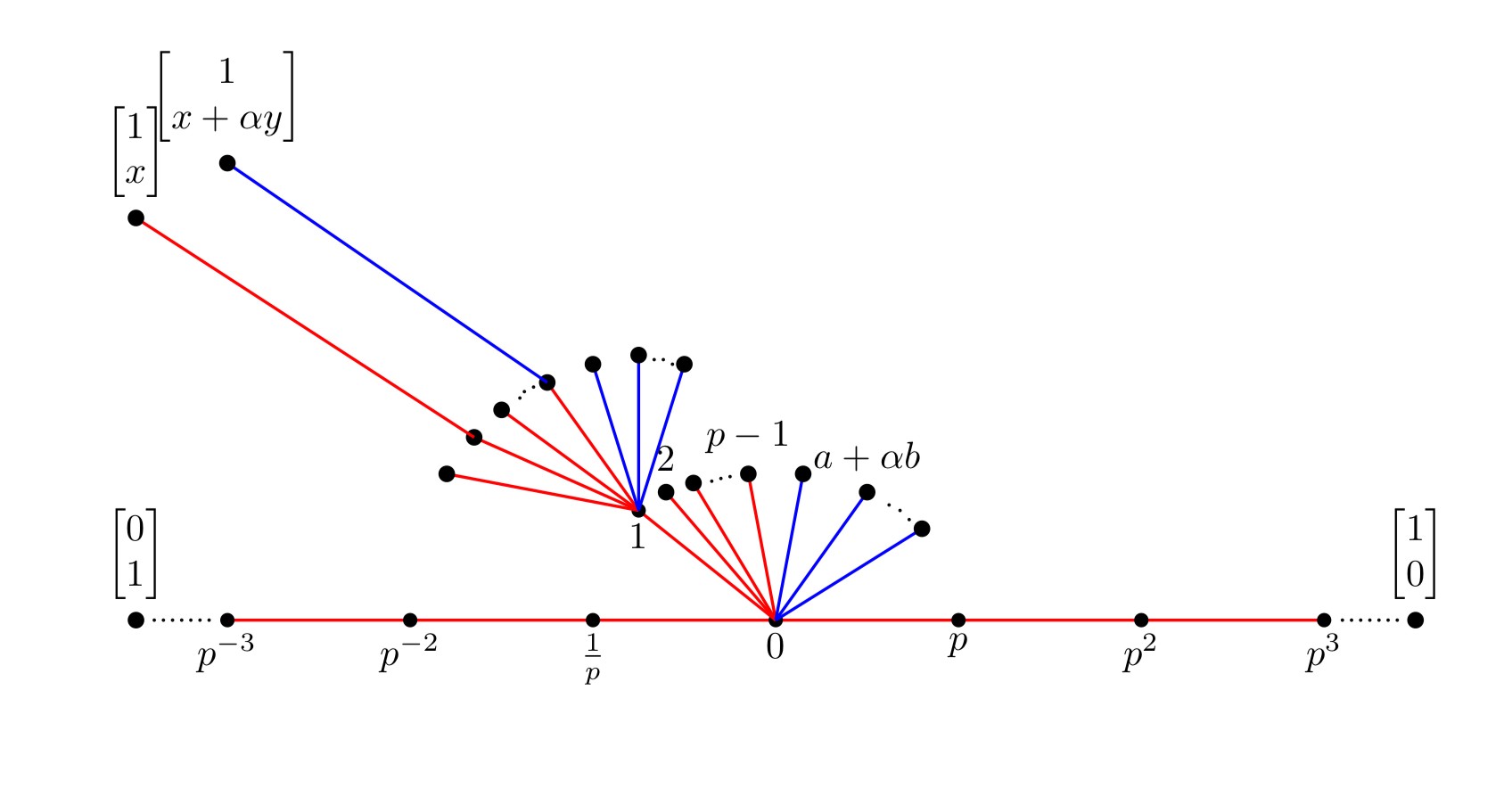}
\caption{Unramified quadratic extension: Let $E = \QQ_p (\alpha)$ where $\alpha^2 = S_p$ and $\alpha \not \in \QQ_p$. 
In the tree for $\SL(2, \QQ_p)$ every vertex has $p+1$ neighbors.  In the tree for $\SL(2,E)$ every vertex has $p^2 +1$ neighbors, obtained by  adding more edges (blue) to each vertex of the tree for $\SL(2,\QQ_p)$. \edit{We} denote $a \in \FF_p$ , $b \in \FF_p ^*$.
}
\label{fig::fig_unram}
\end{figure}

\begin{figure}
\includegraphics[width=12cm, height=6cm]{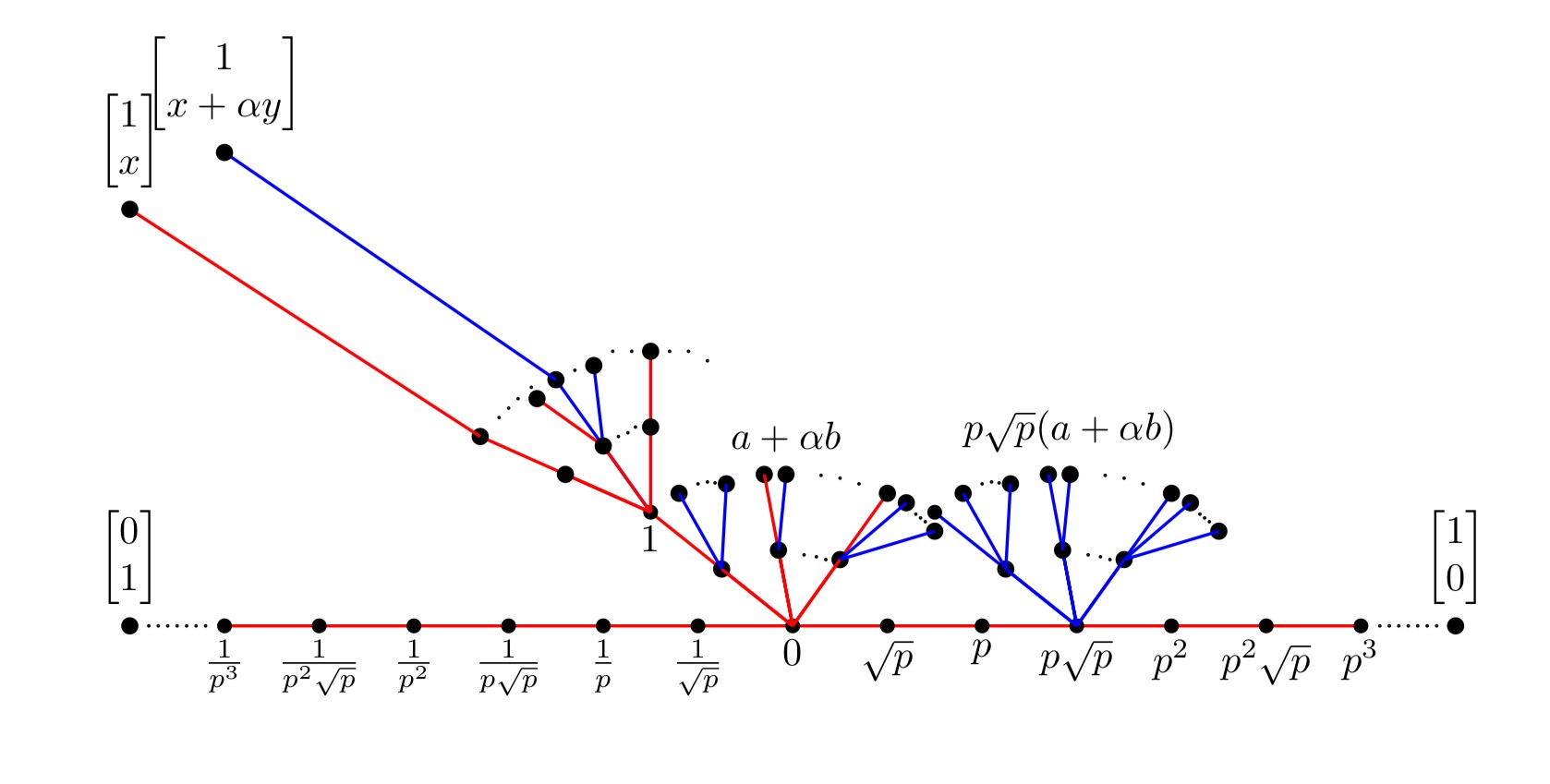}
\caption{Ramified quadratic extension: Let $E = \QQ_p (\alpha)$ where $\alpha^2 \in \{ p, pS_p\}$ and $\alpha \not \in \QQ_p$.  
In the tree for $\SL(2, \QQ_p)$ every vertex has $p+1$ neighbors.  In the tree for $\SL(2,E)$ every vertex also has $p +1$ neighbors, obtained by  adding $p+1$ blue edges from a vertex in the middle of each red edge. \edit{We} denote $a \in \FF_p^{*}$, $b \in \FF_p$.}
\label{fig::fig_ram}
\end{figure}

\medskip
 In the next few paragraphs we summarize results from \cite{HW_class} for $k$-involutions of $\SL(2,k)$ when $k$ is a field of characteristic not \edit{equal to} $2$. Let $\overline{k}$ be the algebraic closure of $k$.

Recall, a mapping $\phi: \SL(2,\overline{k}) \to \SL(2,\overline{k})$ is a \emph{$k$-automorphism} (or equivalently,  \emph{an automorphism defined over $k$}) if $\phi$ is a bijective rational $k$-homomorphism whose inverse is also a rational $k$-homomorphism,  \cite[Sec. 2.2]{Hel_k_invol}. \editA{When $k=\overline{k}$, a $k$-automorphism $\phi$ is called an \emph{algebraic automorphism}, or just an \emph{automorphism}. To distinguish the terminology, an \emph{abstract automorphism} of  $\SL(2,\overline{k})$ is a bi-continuous isomorphism of  $\SL(2,\overline{k})$ to itself, viewed as an abstract group.}

An abstract automorphism $\theta$ of $\SL(2,\overline{k})$ of order two is an \emph{abstract involution} of $\SL(2,\overline{k})$. A \emph{$k$-involution} $\theta$ of $\SL(2,\overline{k})$  is \editA{a $k$-automorphism  of $\SL(2,\overline{k})$ of order two}, and the restriction of $\theta$ to $\SL(2,k)$  is a \emph{$k$-involution} of $\SL(2,k)$. An  \emph{abstract involution} of $\SL(2,k)$ is an abstract automorphism of $\SL(2,k)$ of order two.  Given $g \in \SL(2,\overline{k})$ denote by $\iota_g$ the inner automorphism of $\SL(2,\overline{k})$  defined by $x \mapsto \iota_g(x):= gxg^{-1}$. 

The classification of the isomorphism classes of $k$-involutions of  a connected reductive algebraic group defined over $k$ is given in \cite{Hel_k_invol}. A simple characterization of the isomorphism classes of $k$-involutions of $\SL(n,k)$ is given in \cite{HelWD}.
We record the classification of $k$-involutions of $\SL(2,k)$:

 \begin{theorem}\label{class_invol}[\cite{HW_class} Theorem 1, Corollary 1, Corollary 2]. Every $k$-isomorphism class of $k$-involution of $\SL(2,k)$ is of the form $\iota_A$ with $A= \left( \begin{smallmatrix} 0 &1 \\ a &0 \end{smallmatrix} \right) \in \GL(2,k)$. Two such $k$-involutions $\iota_A$ with $ A\in \{ \left( \begin{smallmatrix} 0 &1 \\ a &0 \end{smallmatrix} \right),  \left( \begin{smallmatrix} 0 &1 \\ b &0 \end{smallmatrix} \right)\} \subset \GL(2,k)$ of $\SL(2,k)$ are conjugate if and only if $a$ and $b$ are in the same square class of $k^{*}$. In particular, there are $order(k^* / (k^ *)^2)$ $k$-isomorphism classes of $k$-involutions of $\SL(2,k)$. 
 \end{theorem}

\begin{definition} Given an involution $\theta$ of a group $G$ the \emph{fixed point group} of $\theta$ is $H_{\theta}: = \{ x \in G \; \vert \; \theta(x) = x \}$. 
\end{definition} 

For $\theta$ a $k$-involution of $\SL(2,k)$ the quotient $\SL(2,k)/H_{\theta}$ is called a \emph{$k$-symmetric variety}, and much of the structure of $\SL(2,k)/H_{\theta}$ is determined by $H_\theta$. 

\begin{proposition}[\cite{HW_class} Section 3]
\label{prop::prop_shape_inv}
 Let $\theta= \iota_A$, with $A= \left( \begin{smallmatrix} 0 &1 \\ a &0 \end{smallmatrix} \right)  \in \GL(2,k)$, be a $k$-involution of $\SL(2,k)$. Then $ H_{\theta}=\{  \left( \begin{smallmatrix} x &y \\ ay &x\end{smallmatrix} \right) \in \SL(2,k)\; \vert \; x^2 - ay^2 =1 \} $. 
\end{proposition}

\editA{A quadratic form $q$ is \emph{isotropic} if there exists a vector $x$ such that $q(x)=0$.  Otherwise $q$ is \emph{anisotropic}. In the context of groups, a non-compact subgroup $H_{\theta}$ will be called \emph{isotropic} \editB{when $\theta$ is isotropic}, and a bounded (or compact) subgroup $H_{\theta}$ will be called \emph{anisotropic} \editB{ when $\theta$ is anisotropic}. %This will correspond to involutions $\theta$ coming from an/isotropic forms.
}

\begin{theorem} [\cite{HW_class} Section 3.2]
\label{thm::thm_anist_iso} Let $k=\QQ_p$, $\theta = \iota_A$ with $A= \left( \begin{smallmatrix} 0 &1 \\ a &0 \end{smallmatrix} \right)$ and $\overline{a} \in \QQ_{p}^{*} / (\QQ_p^*)^2$.  Then $H_\theta$ is anisotropic if and only if $\bar a \neq \bar 1$.  If $\bar a = \bar 1$, then $H_\theta$ is isotropic and conjugate to the maximal $\QQ_p$-split torus of $\SL(2,\QQ_p)$, i.e. the diagonal subgroup of $\SL(2,\QQ_p)$. 
\end{theorem} 

\begin{remark}
\label{rem::inv_trans}
In the case of $2\times 2$ matrices the operation inverse composed with transpose is given by an inner automorphism:  
\[ \left(\left( \begin{matrix} a &b \\ c &d \end{matrix} \right)^{T}\right)^{-1}= \left( \begin{matrix} a &c \\ b &d \end{matrix} \right)^{-1}= \left( \begin{matrix} d &-c \\ -b &a \end{matrix} \right)  = \left( \begin{matrix} 0& 1\\ -1 &0 \end{matrix} \right) \left( \begin{matrix} a &b \\ c &d \end{matrix} \right) \left( \begin{matrix} 0& -1\\ 1 &0 \end{matrix} \right) .\]
This is not true for matrices of higher rank.
\end{remark}

Finally, we link the $\QQ_p$-involutions of $\SL(2,\QQ_p)$ given by Theorem \ref{class_invol} with quadratic forms over $\QQ_p$.
\begin{remark}
\label{rem::quadratic_forms} 
By \cite[Corollary of Section 2.3]{Serre}, for $p \neq 2$, there are exactly 7 classes of quadratic forms of rank $2$ over $\QQ_p$. We apply Remark \ref{rem::inv_trans}. As $\left( \begin{smallmatrix} 1& 0\\ 0 &-a \end{smallmatrix} \right)\left( \begin{smallmatrix} 0& 1\\ -1 &0 \end{smallmatrix} \right)= \left( \begin{smallmatrix} 0& 1\\ a &0 \end{smallmatrix} \right)=:A\in \GL(2,\QQ_p)$,  an involution $\theta=\iota_A$ of $\SL(2,\QQ_p)$ is determined by the quadratic form \edit{associated to the symmetric matrix $A$} %$\left( \begin{matrix} 1& 0\\ 0 &-a \end{matrix} \right)$:
\[\iota_A(g)=AgA^{-1	}= \left( \begin{matrix} 1& 0\\ 0 &-a \end{matrix} \right)(g^{-1})^{T} \left( \begin{matrix} 1& 0\\ 0 &-a \end{matrix} \right)^{-1}\]
for every $g \in \SL(2,\QQ_p)$.

For $\SL(n,\QQ_p)$, with $n\geq 3$, the number of $\QQ_p$-involutions is strictly larger than the number of the quadratic forms of rank $n$ over $\QQ_p$, the latter is given by  \emph{outer} $\QQ_p$-involutions (see \cite[Section 4.1.7, Section 6]{HelWD}). 
\end{remark}

 %Benoist and Oh give a $KAH$ decomposition where $K$ and $A$ are not necessarily subgroups. 

%Let $k$ be a non-Archimedean local field of characteristic not 2. Let $G$ be a connected reductive $k$-group, $\theta$ a $k$-involution of $G$ and $H$ an open $k$-subgroup of the group $G^\theta$ of $\theta$-fixed points.  Since there are only finitely many classes of $(k, \theta)$ split tori, we may choose a set of representatives $\{ A_i | 1 \leq i \leq n \}$ of $H$- conjugacy classes of maximal $(k, \theta)$ split tori of $G$, and set $A = \cup_{i=1}^n A_i$. 

%\begin{theorem}\label{BenOh}[\cite{BenoistOh} Theorem 1.1] There exists a compact subset $K \subset G$ so that $G= KAH$. 
%\end{theorem} 

%Notice that in general $K$ cannot be chosen to be a subgroup, and the union of the tori $A$ is also not a subgroup.  However, this theorem \emph{reduces our problem to only computing limits where we conjugate by elements in $A$.} Thus we may use it similarly to the reduction made using a $KAH$ decomposition for a real group, which allows us to only compute limits of $H$ by conjugating by elements of $A$. 

\section{Automorphisms and abstract automorphisms of $\SL(2,\QQ_p(\alpha))$}\label{sec:autom}%%%%%%%%%%%%%

Let $k$ be a local field of characteristic not \edit{equal to} $2$ and denote by $\overline{k}$ the algebraic closure of $k$. The group of $k$-automorphisms of $\SL(2,\overline{k})$ is denoted by $\Aut_{k}(\SL(2,\overline{k}))$. If $k=\overline{k}$ then we just write  $\Aut(\SL(2,\overline{k}))$, see \cite[Sec. 2.2]{Hel_k_invol}, \editA{and those are called \emph{algebraic automorphisms}, or just \emph{automorphisms}, of $\SL(2,\overline{k})$.}

Denote the group of inner automorphisms of $ \SL(2,\overline{k})$ by $\Inn( \SL(2,\overline{k}))$  and the group of inner $k$-automorphisms of $ \SL(2,\overline{k})$ by $\Inn_{k}(\SL(2,\overline{k}))$. Then $\Inn_{k}( \SL(2,\overline{k}))= \Inn( \SL(2,\overline{k})) \cap\Aut_{k}(\SL(2,\overline{k}))$,  see \cite[Sec. 2.2]{Hel_k_invol}. %Note that $\Inn_k (\SL(2,\overline{k})) \supsetneq \Inn( \SL(2,k)):=\{\iota_g \;\vert\; g \in  \SL(2,k))\}$, as $\iota_g \in \Inn_k (\SL(2,\overline{k})) \setminus \Inn( \SL(2,k))$, for every $g \in \GL(2,k)$ {\red ???}.

By Borel \cite{Bo} we have $\Aut(\SL(2,\overline{k}))=\Inn( \SL(2,\overline{k}))$.

We denote by $\Aut_{\abstrc}(\SL(2,\overline{k}))$ the group of all abstract automorphisms of  $\SL(2,\overline{k})$ \editA{(see the terminology introduced just above Theorem \ref{class_invol})},  and by $\Aut(k)$ the group of all bi-continuous field automorphisms of $k$. 

\medskip
Abstract (bi-continuous) automorphisms of the $k$-rational points $G(k)$ of an absolutely almost simple algebraic group $G$ defined over an infinite field $k$ were described by Borel--Tits \cite{BoTi}, and also by Steinberg \cite{Stein}. For example, the group $\SL(n,k)$ is  an absolutely almost simple, simply connected group, and \edit{splits} over $k$.  

By those \edit{results}, for $G$ as mentioned above, we have that $\Aut_{\abstrc}(G(k))$ fits in the exact sequence 
 \[1 \to \Aut_{k}(G) \to \Aut_{\abstrc}(G(k)) \to \Aut(k).\] 
Let $\Aut_G(k)$ be the image of $\Aut_{\abstrc}(G(k))$ in $\Aut(k)$.  When $G$ is a $k$-split connected reductive group, then $\Aut(k)= \Aut_G(k)$. And when $G$ is $k$-split, $\Aut_{\abstrc}(G(k))$ splits as the semi-direct product $\Aut_{k}(G) \rtimes  \Aut_G(k)$, (see \cite[Section 9.1]{BCL}, \cite[Introduction]{Stu}). 

\medskip
Thus, for the particular case of $\SL(2,k)$ we have that $\Aut_{\abstrc}(\SL(2,k))= \Aut_{k}(\SL(2,\overline{k})) \rtimes  \Aut(k)$. In \cite{Hel_k_invol, HW_class} only $k$-involutions of $\SL(2,k)$ are studied, i.e. involutions in $\Aut_{k}(\SL(2,\overline{k}))$. 

In order to obtain all abstract involutions of $\SL(2,k)$, thus involutions in $\Aut_{\abstrc}(\SL(2,k))$, it remains to compute $\Aut(k)$ and then to combine  with the $k$-automorphisms. 

\medskip
For $k=\QQ_p$, by \cite{Con} we have $\Aut(\QQ_p)=\{\Id\}$, and so for $\SL(2,\QQ_p)$ we have only $\QQ_p$-involutions and those are computed by the results in \cite{HW_class} recalled in Section \ref{sec:background} above. 

\medskip
For $k=E=\QQ_p(\alpha)$ a quadratic extension of $\QQ_p$,  to compute all the abstract involutions of $\SL(2,E)$,  we compute $\Aut(E)$ and  then combine \edit{it} with  $\Aut_E(\SL(2,\overline{\QQ_p}))$ in order to obtain all abstract involutions of $\SL(2,E)$. 

\begin{definition}
\label{def::conj}
Let $k$ be a local field and $K$ a finite field-extension of $k$. Let $\theta, \Psi \in \Aut_{\abstrc}(\SL(2,k))$. We say that $\theta$ and $\Psi$ are \emph{$\GL(2,K)$-conjugate} if there is $X \in \Inn(\GL(2,K))$,  such that $X^{-1} \theta X =\Psi$. In particular, there is a matrix $A \in \GL(2,K)$, such that $X = \iota_A$, and this means $A^{-1}(\theta(A g A^{-1}))A = \Psi(g)$, for every $g \in \SL(2,k)$.
\end{definition}

%\blue{
%\begin{definition}
%\label{def::conj}
%Let $k$ be a local field. Let $\theta, \Psi \in \Aut_{abstrc}(\SL(2,k))$. We say that $\theta$ and $\Psi$ are $\SL(2,k)$-conjugated, if there is $X \in \Inn(\SL(2,k))$,  such that $X^{-1} \theta X =\Psi$. In particular, there is a matrix $A \in \SL(2,k)$, such that $X = \iota_A$, and this means $A^{-1}(\theta(A g A^{-1}))A = \Psi(g)$, for every $g \in \SL(2,k)$. A similar definition can be recorded for \emph{$\GL(2,k)$-conjugacy.} 
%\end{definition}
%} 

\begin{lemma}  Let $E$ be a quadratic extension of $\QQ_p$, so $E=\QQ_p (\alpha)$, where $\alpha \in E\setminus \{\QQ_p\} $ and $\alpha ^2 \in \QQ_p$.  Then $\Aut(E)=\{\Id,\sigma\}$ where $\sigma( a + \alpha b) = a - \alpha b$. 
\end{lemma} 

\begin{proof} 
From \cite{Schm} we know that a field which is complete with respect to two inequivalent nontrivial norms (i.e., the two norms induce distinct non-discrete topologies) must be algebraically closed.  A corollary is that a field which is complete with respect to a nontrivial norm and which is not algebraically closed has only one equivalence class of norm, $n$.  So  an automorphism $\sigma$ of $E$  induces a norm $n_\sigma (x) := n( \sigma(x))$.  But then $n_\sigma$ must be equivalent to the unique norm $n$, so $n_\sigma$ is some scalar multiple of $n$. Thus every automorphism of $E$ is continuous \edit{on $E$} with respect to the norm topology. %If we apply this to our quadratic extension $E=\QQ_p(\alpha)$ of $\QQ_p$ we get that every automorphism of $E$ is continuous. 
Since $\QQ$ is dense in $\QQ_p$ and because any automorphism of $E$ is continuous and the identity on $\QQ$, one can deduce that any \edit{automorphism} on $E$ is trivial on $\QQ_p$, thus Galois. Now as $\alpha^2 \in \QQ_p$, we have that any automorphism of $E$ will send $\alpha$ to $\pm\alpha$. Therefore $\Aut(E)=\{\Id,\sigma\}$, where  $\sigma: \QQ_p(\alpha) \to \QQ_p(\alpha)$, with $x \in \QQ_p \mapsto \sigma(x)=x$, and $\alpha \mapsto \sigma(\alpha)=-\alpha$.
\end{proof} 

\begin{remark}
\label{rem::conj_inv_auto} Let $k$ be a local field. 
By the results of Borel--Tits \cite{BoTi} that are recalled in \cite[Theorem 9.1 v)]{BCL} we know that $\Aut_{\abstrc}(\SL(2,k))$ acts continuously, properly and faithfully on the Bruhat--Tits tree of $\SL(2,k)$. In the particular case when $k=E$ is a quadratic extension of $\QQ_p$, the involution $\sigma$ is an automorphism of the Bruhat--Tits tree of $\SL(2,E)$, as well as any abstract involution $\theta$ of $\SL(2,E)$.
\end{remark}

Let us now compute the abstract involutions of $\SL(2,E)$.  \editA{If $A \in \SL(2,E)$ is a matrix and $\gamma \in \Aut(E)$ is a field automorphism, then $\gamma(A)$ is the matrix where $\gamma$ is applied to every matrix entry. }

\begin{theorem}
\label{thm::inv_SL_E}
Let  $E=\QQ_p(\alpha)$ be a quadratic extension of $\QQ_p$. Then any abstract involution $\theta$ of $\SL(2,E)$ is of the form $ \theta = \iota_A \circ \gamma$ where: 
\begin{enumerate}
\item
 \edit{either} $\gamma= \sigma$  and $A \in \edit{ \left \{\left( \begin{smallmatrix} z &y \\ 1 & -\sigma(z) \end{smallmatrix} \right),  \left( \begin{smallmatrix} x &0 \\ 0 & 1 \end{smallmatrix} \right)\right \}} \subset \GL(2,E)$, with $y \in \QQ_p, z,x\in E$, with $z\sigma(z)+ y \neq 0$ and $x\sigma(x)=1,$
\item
 \edit{or} $\gamma=\Id$ and \edit{ $A$ is $\SL(2,E)$-conjugate to a matrix of the form} $\left( \begin{smallmatrix} 0 &1 \\ a &0 \end{smallmatrix} \right)$, with $a\in E^{*}/(E^{*})^{2}$.
\end{enumerate}
\end{theorem}

\begin{proof} 
By Borel--Tits \cite{BoTi} any abstract automorphism of $\SL(2,E)$  is written as $\beta \circ \gamma$, with $\beta \in \Aut_E(\SL(2,\overline{\QQ_p}))$ and $\gamma \in \Aut(E)=\{\Id,\sigma\}$.  By  \cite[Remark 2]{HW_class} every $E$-automorphism of $\SL(2,E)$ can be written as the restriction to $\SL(2,E)$ of a inner automorphism of $\GL(2,E)$. Thus, given $\theta$ an abstract automorphism of $\SL(2,E)$ there exists $A =  \left( \begin{smallmatrix} a &b \\ c &d \end{smallmatrix} \right) \in \GL(2,E)$ such that $\theta = \iota_A \circ \gamma$. From now on assume that $\theta$ is an abstract involution, thus $\theta^2 = \Id$. Then, by \cite[Lemma 2]{HW_class}, $A\gamma(A) = \left( \begin{smallmatrix} a\gamma(a) + b\gamma(c)&a\gamma(b)+b\gamma(d) \\ c\gamma(a)+d\gamma(c) & c\gamma(b)+d\gamma(d) \end{smallmatrix} \right) = q\Id$, for some $q \in E^*$, and with $\gamma \in \{\Id, \sigma\}$.  Thus we have 
\begin{equation}
\label{equ::rel_A}
a \gamma(a) + b\gamma(c) = c \gamma (b) + d\gamma(d) \qquad \text{and}  \qquad c\gamma(a) + d \gamma(c) = 0 = a \gamma(b) + b \gamma(d).
\end{equation}

When $\gamma = \Id$ we can directly apply the results of \cite[Section 1.3]{HW_class} and get that up to $\SL(2,E)$-conjugacy,  $\theta = \iota_A$, with  $A= \left( \begin{smallmatrix} 0 &1 \\ t &0 \end{smallmatrix} \right)$, with $t \in E^{*}/(E^{*})^{2}$. 

We consider now the case $\gamma=\sigma$. By applying $\sigma$ to the second equality of (\ref{equ::rel_A}) and putting together all four equations, we get 
\begin{equation}\label{rel_A2} a (\sigma(b) - \sigma(c) ) = \sigma(d) (c-b) \qquad \text{and}  \qquad  \sigma(a) (b-c) = d (\sigma(c) - \sigma(b)). \end{equation}

\textbf{Case 1: \edit{If} $ a \neq 0$}  then \editA{by \eqref{equ::rel_A}} $d \neq 0$ and \editA{taking the ratio of the equations in \eqref{rel_A2} when it makes sense, or using the first equality of \eqref{equ::rel_A},} gives $0\neq \frac{a}{d}= \frac{\sigma(d)}{\sigma(a)} \editA{:=} x \in E$ , so $a = xd$ and $\sigma(d) = x \sigma(a)$. By applying $\sigma$, we \edit{have} also $\sigma(a) = \sigma(x) \sigma(d) $.  Then $\sigma(a) = \sigma(x) x \sigma(a)$ thus getting $\sigma(x) x =1$.  Replacing \edit{$a$ by $xd$} in the first equality of (\ref{equ::rel_A}) we get $b\sigma(c)=c\sigma(b)$.

If $c=b=0$ then  $A= \left( \begin{smallmatrix} xd &0 \\ 0&d \end{smallmatrix}\right)$, and \editA{modding out by the center} one can just take $A= \left( \begin{smallmatrix} x &0 \\ 0&1 \end{smallmatrix}\right)$. Then \editB{one can verify that indeed} $A\sigma(A)=  \left( \begin{smallmatrix} x\sigma(x) &0 \\ 0&1 \end{smallmatrix}\right)=q\Id$, implying \editA{$q=1$}. 

As $b\sigma(c)=c\sigma(b)$ is symmetric in $b$ and $c$, we consider next that $c \neq 0$, then $b= cy$, with  $y \in E$. Replacing $b= cy$ in $b\sigma(c)=c\sigma(b)$ one gets $y=\sigma(y)$, so $y \in \QQ_p$. Thus $A= \left( \begin{smallmatrix} xd &cy \\ c&d \end{smallmatrix}\right) $ and so  
\[A\sigma(A) 
= \begin{pmatrix} 
d\sigma(d)  + yc\sigma(c) & xdy\sigma(c) + cy\sigma(d) \\
c \sigma(xd) + d \sigma(c) & cy \sigma(c) + d \sigma(d) 
\end{pmatrix} =
\begin{pmatrix} q&0\\0 &q
\end{pmatrix}  \] 
from which it follows 
\[xd \sigma(c) =-c\sigma(d) \qquad  d\sigma(d) + yc\sigma(c) =q,\] 
 \edit{ implying $q \in \QQ_p$, because $y \in \QQ_p$}. We have three equations 
\[ xd\sigma(c) = -c\sigma(d) \qquad \sigma(x) \sigma(d) c= -\sigma(c)d \qquad d \sigma(d)= q-yc\sigma(c).\] 
As by our assumption $c \neq 0$ and $d \neq 0$ then $ 0 \neq \frac{xd}{c}= \frac{-\sigma(d) }{\sigma(c) }=: z \in E$.  Rewriting our equations to include $z$ we have 
\[ xd = zc \qquad \sigma(x) \sigma(d) = \sigma(z) \sigma(c) \qquad - \sigma(d) = \sigma(c)z \qquad -d = c\sigma(z).\] 
Thus $- \sigma(x) z = \sigma(z)$ and we can write $\sigma(x) = \frac{-\sigma(z)}{z}$.  Returning to $A$ and substituting from above $d= \frac{zc}{x}$, we have 
\[ A= \begin{pmatrix} xd &cy \\ c&d \end{pmatrix} 
=
\begin{pmatrix} x\frac{zc}{x} &cy \\ c&\frac{zc}{x}\end{pmatrix} 
= c 
\begin{pmatrix} z&y \\ 1 & - \sigma(z) \end{pmatrix} .
\] 
And returning to our computation \edit{for} $A\sigma(A)$ and taking $q' := \frac{q}{c \sigma(c)} \in \QQ_p$ \editB{one can verify that indeed}
\[\frac{1}{c\sigma(c)} A\sigma(A)= \begin{pmatrix}z \sigma(z) +y &0 \\ 0 & y + \sigma(z)z \end{pmatrix}
=  \begin{pmatrix} \editA{q'}&0\\0 &\editA{q'}
\end{pmatrix}, \qquad q,y \in \QQ_p.\] 

%If $b=0$, then $A = \left( \begin{smallmatrix} xd& 0 \\ c&d \end{smallmatrix} \right) $, and we can reduce $A$ to  $A = \left( \begin{smallmatrix} x & 0 \\ \edit{c'} & 1 \end{smallmatrix} \right) $, for some $\edit{c'} \in E^{*}$.  Taking 
%\[A\sigma(A) = \begin{pmatrix} 
%x \sigma(x) & 0 \\ \edit{c'} \sigma(x) +  \sigma(\edit{c'}) & 1
%\end{pmatrix} 
%= \begin{pmatrix} \editA{q'}&0\\0 &\editA{q'}
%\end{pmatrix}  \]
%we obtain the condition $\edit{c'} \sigma(x) +  \sigma(\edit{c'})=0$. Since $x\sigma(x)=1$, \editA{we get $c'=-x\sigma(c')$}. \editA{Since $\iota_{A}$ and $\iota_{\frac{1}{c'}A}$ give the same involution $\theta$, one can replace $A=  \left( \begin{smallmatrix} x & 0 \\ c' & 1 \end{smallmatrix} \right)$ by $\frac{1}{c'}A=  \left( \begin{smallmatrix} x/c' & 0 \\ 1 & x\sigma(x)/c' \end{smallmatrix} \right) =  \left( \begin{smallmatrix} x/c' & 0 \\ 1 & -\sigma(x)/\sigma(c') \end{smallmatrix} \right) $}. 

\editB{Notice that the case $b=0$ follows from the case when $c \neq 0$ and $b= cy$ by taking $y=0$.} This completes the proof if $a \neq 0$.

\textbf{Case 2: \edit{If} $a=0$} then $d=0$, and $A= \left( \begin{smallmatrix}0 &b \\ c&0 \end{smallmatrix}\right) $, and from (\ref{equ::rel_A})  $b \sigma(c) = c \sigma(b)$.  So $\frac{b}{c}= \frac{\sigma(b)}{\sigma(c)} = :y \neq 0$, and thus $y\in \QQ_p^{*}$.  Then $A= \left( \begin{smallmatrix}0 &cy \\ c&0 \end{smallmatrix}\right) $, reducing to $A= \left( \begin{smallmatrix}0 &y \\ 1&0 \end{smallmatrix}\right) $ with $y \in \QQ_p^{*}$.
\end{proof}

Now we compute the fixed point groups of the involutions from Theorem \ref{thm::inv_SL_E} using the ends of the Bruhat--Tits tree of $\SL(2,K)$, where $K$ is \edit{a} finite field-extension of $E$, and thus of $\QQ_p$, that is chosen suitably. We will show those fixed point groups are either trivial, or compact, or $\GL(2,E)$-conjugate to either the diagonal subgroup in $\SL(2, E)$, or to $\SL(2, \QQ_p)$. Recall $E= \QQ_p(\alpha)$ is a quadratic extension of $\QQ_p$ and $\mathcal{O}_E$ denotes the ring of integers of $E$. % {\blue From \cite{Serre} for $K$ a local field,  the vertices of the Bruhat--Tits tree for $\SL(2,K)$ are equivalence classes of $\mathcal{O}_K$-lattices in $K^2$. }

We start with some easy lemmas.
\begin{lemma}
\label{lem::conj_cond}
Let $X \in \SL(2,E)$, \edit{or} $X \in \GL(2,E)$. Let $\theta= \iota_A \circ \sigma$ be an abstract involution of $\SL(2,E)$, with $A \in \GL(2,E)$ as in Theorem \ref{thm::inv_SL_E}. Then $\iota_{X^{-1}} \circ \theta \circ \iota_{X}= \sigma$ if and only if $A\sigma(X)=qX$, for some $q \in E^{*}$.
\end{lemma}
\begin{proof}
By writing the equality $\iota_{X^{-1}} \theta \iota_{X}= \sigma$, we have 
\begin{equation*}
\begin{split}
\sigma(g)&= \iota_{X^{-1}} \theta \iota_{X}(g)= X^{-1}(A(\sigma(XgX^{-1}))A^{-1})X= X^{-1}A\sigma(X)\sigma(g)\sigma(X^{-1})A^{-1}X\\
&=\iota_{X^{-1}A\sigma(X)}(\sigma(g)), \text{ for every } g \in \SL(2,E).
\end{split}
\end{equation*}
Then by \cite[Lemma 2]{HW_class}, we have $X^{-1}A\sigma(X)= q\Id$ for some $q \in E^{*}$. The converse implication is  trivial.
\end{proof}

\begin{definition}
Let $T_d$ be a $d$-regular tree with $d \geq 3$.   Denote by $\partial T_d$ the visual boundary or ends \editA{which are identified for} $T_d$. An automorphism $\theta_1 \in \Aut(T_d)$ is  a \emph{tree-involution} if $\theta_1^{2} = \Id \in \Aut(T_d)$, in particular $\theta_1$ is an elliptic automorphism of $T_d$. 

We denote by $\Fix_{T_d}(\theta_1):=\{ v \in T_d \; \vert \; \theta_1(v)=v \}$, and by $\Fix_{\partial T_d}(\theta_1):=\{ \xi \in \partial T_d \; \vert \; \theta_1(\xi)=\xi \}$. Notice, $\Fix_{T_d}(\theta_1)$ is a (finite or infinite) connected subtree of $T_d$, and when the latter is infinite, we have $\Fix_{\partial T_d}(\theta_1) \subseteq \partial \Fix_{T_d}(\theta_1)$.  

Also, consider the following two subgroups of $\Aut(T_d)$:
$$\Stab_{\Aut(T_d)}(\Fix_{T_d}(\theta_1)):= \{ g \in \Aut(T_d) \; \vert \; g(\Fix_{T_d}(\theta_1))= \Fix_{T_d}(\theta_1) \text{ setwise}\} \text{ and  }$$
$$\Stab_{\Aut(T_d)}(\Fix_{\partial T_d}(\theta_1)):= \{ g \in \Aut(T_d) \; \vert \; g(\Fix_{\partial T_d}(\theta_1))= \Fix_{\partial T_d}(\theta_1) \text{ setwise}\}.$$

If moreover $\Fix_{\partial T_d}(\theta_1)$ is not the empty set we have 
$$\Stab_{\Aut(T_d)}(\Fix_{T_d}(\theta_1)) \leq \Stab_{\Aut(T_d)}(\Fix_{\partial T_d}(\theta_1)) \text{ but they might not be equal.}$$ 
\end{definition}

\begin{lemma}
Let $K$ be a finite field-extension of $\QQ_p$. Let $\theta$ be an abstract involution of $\SL(2,K)$, and let $H_{\theta}:=\{g \in \SL(2,K) \; \vert \; \theta(g)= g\}$. Let $\theta_{1}$ be the automorphism of the Bruhat--Tits tree $T_K$ of $\SL(2,K)$ induced by $\theta$. Then $\theta_1$ is a tree-involution of $T_K$ and 
$$H_{\theta} \leq \Stab_{\Aut(T_K)}(\Fix_{T_K}(\theta_1)) \leq \Stab_{\Aut(T_K)}(\Fix_{\partial T_K}(\theta_1)).$$
\end{lemma}
\begin{proof}
Since $\theta^{2} = \Id$, and by Remark \ref{rem::conj_inv_auto} we have that $\theta_{1}^{2}=\Id$, thus $\theta_{1}$ is an involution in $\Aut(T_K)$ as claimed, and in particular an elliptic automorphism of $T_K$.

Take now $g \in H_\theta$ and $x \in \Fix_{T_K}(\theta_1)$. Then $\theta_1(g(x))= \theta(g)(\theta_1(x))=g(x)$, and thus $g(x) \in  \Fix_{T_K}(\theta_1)$, implying that $g \in \Stab_{\Aut(T_K)}(\Fix_{T_K}(\theta_1))$ and the conclusion follows. 
\end{proof}

We first want to understand the case of the involution $\theta=\iota_A$ of $\SL(2,E)$  from Theorem \ref{thm::inv_SL_E}
given by the matrix  $A=\left( \begin{smallmatrix} 0 &1 \\ a &0 \end{smallmatrix} \right)$, with $a\in E^{*}/(E^{*})^{2}$.  The fixed point group $H_\theta \leq \SL(2,E)$ of $\theta$ is computed in \cite{HW_class}.   We give a geometric interpretation of $H_\theta$ using the ends of the Bruhat--Tits tree of $\SL(2,K_a)$, where $K_a:=E(\sqrt{a})$ is  a  quadratic extension of $E$.  
The ends of the tree $T_{K_a}$ are the elements of the projective space $P^{1}K_a= \big[ \begin{smallmatrix} 0 \\1  \end{smallmatrix} \big]   \bigsqcup\limits_{x\in K_a} \big[ \begin{smallmatrix} 1 \\x  \end{smallmatrix} \big]$.   Notice, the action of $\theta = \iota_A$ on the Bruhat--Tits tree (and its boundary) of $\SL(2,E)$, and respectively of $\SL(2,K_a)$, is the \editA{automorphism} induced by $A \in \GL(2,E) \editA{\leq \Aut(T_E)}$ on the trees and their respective boundaries.  \editA{For example, for $A= \left(\begin{matrix} x&y \\ z &t \end{matrix} \right) \in \GL(2,E)$ and $\xi = \big[ \begin{smallmatrix} c \\d  \end{smallmatrix} \big] \in P^{1}E$ an end, $A(\xi)=\left(\begin{matrix} x&y \\ z &t \end{matrix} \right) \big[ \begin{smallmatrix} c \\d  \end{smallmatrix} \big]= \big[ \begin{smallmatrix} xc +yd \\zc +td  \end{smallmatrix} \big] $.}

\editB{Given $X \subset P^1K$, we denote by $\Fix_{\SL(2, K)}(X) := \{ g \in \SL(2,K) \;|\; g(x) = x, \forall x \in X \} $ the set of elements in $\SL(2, K)$ fixing $X$ pointwise. }

%Notice, the action of $\theta = \iota_A$ on the Bruhat--Tits tree (and its boundary) of $\SL(2,E)$, and respectively of $\SL(2,K_a)$, is the natural map induced by $A \in \GL(2,E)$ on  the trees and their respective boundaries. \todo{Ref says recall $\theta \cdot \xi$, but I am not sure this is what we want}  Now we will use the ends of the Bruhat--Tits trees to understand the fixed point groups of the abstract involutions. \editA{Given a subset $X \subset P^1 K_a$, the fixed point set $\Fix_{\SL(2,K_a)} (X) = \{A \in \SL(2, K_a) | A(x) =x, \forall x \in X \}$. }

\begin{corollary}
\label{lem::geom_k_inv_SL_2}
Let $F$ be a finite field-extension of $\QQ_p$, $A=\left( \begin{smallmatrix} 0 &1 \\ a &0 \end{smallmatrix} \right)$, with $a\in F^{*}/(F^{*})^{2}$, and $\theta_a:=\iota_A$ the corresponding $F$-involution of $\SL(2,F)$.  Take $K_a:=F(\sqrt{a})$  a quadratic field extension. %of $F$ corresponding to $a$. 
Then the only solutions of the equation $A(\xi)=\xi$ with $\xi \in P^{1}K_a$ are   $\xi_{\pm} := \big[ \begin{smallmatrix} 1 \\ \pm \sqrt{a} \end{smallmatrix} \big]$ and $H_{\theta_a} = \Fix_{\SL(2,K_a)}(\{\xi_{-},\xi_{+}\}) \cap \SL(2,F) =\{  \left( \begin{smallmatrix} x &y \\ ay &x\end{smallmatrix} \right) \in \SL(2,F)\; \vert \; x^2 - ay^2 =1 \} $. Moreover,
\begin{enumerate}
\item
if $a=1$ then  $\xi_{\pm} :=\big[ \begin{smallmatrix}  1 \\ \pm 1 \end{smallmatrix} \big]$  and  $H_{\theta_a}$ contains all the hyperbolic elements of $\SL(2,F)$ with $\xi_{\pm}$ as their repelling and attracting endpoints. In particular, $H_{\theta_a}$ is  $\GL(2,F)$-conjugate to the entire diagonal subgroup of $\SL(2,F)$,  thus $H_{\theta_a}$ is noncompact and abelian. 
\item
if $a\neq 1$ then  $\xi_{\pm} :=\big[  \begin{smallmatrix}  1 \\ \pm  \sqrt{a}  \end{smallmatrix} \big]  \in P^{1}K_a - P^{1}F$, and $H_{\theta_a}$ is compact and abelian.
\end{enumerate}
\end{corollary}
\begin{proof}
We search for all the ends $\xi \in P^{1}K_a$ such that $A(\xi)=\xi$. This means that we search for all vectors  $ \big(  \begin{smallmatrix} 1 \\ x \end{smallmatrix} \big)$ with $x \in K_a$, or $\big( \begin{smallmatrix} 0 \\ 1 \end{smallmatrix} \big)$, such that:
 \[ \left(\begin{matrix} 0 &1 \\ a &0 \end{matrix} \right)  \begin{pmatrix} 1 \\ x \end{pmatrix} = c  \begin{pmatrix} 1 \\ x \end{pmatrix} \editA{\textrm{i.e.,}}  \begin{pmatrix} x \\ a \end{pmatrix} =   \begin{pmatrix} c \\ cx \end{pmatrix} 
\qquad \textrm{or} \qquad 
\left(\begin{matrix} 0 &1 \\ a &0 \end{matrix} \right)  \begin{pmatrix} 0\\ 1 \end{pmatrix} = c  \begin{pmatrix} 0 \\ 1 \end{pmatrix} \editA{\textrm{i.e.,}}  \begin{pmatrix} 1 \\ 0\end{pmatrix} = \begin{pmatrix} 0 \\ c \end{pmatrix}\] 
\text{ for some } $c \in K_a^{*} $. 
One can see that the only solutions are given by $c=x$ and $a=c^2$, implying that $\xi_{\pm} := \big[ \begin{smallmatrix} 1 \\ \pm \sqrt{a} \end{smallmatrix} \big]  \in P^{1}K_a= P^{1}F(\sqrt{a})$ are the only solutions of the equation $A(\xi)=\xi$. 

Next, we claim that the involution $\theta_a =\iota_A$ is  $\GL(2,K_a)$-conjugate to the involution $\iota_{B}$ of $\SL(2,K_a)$, where  $B:= \left(\begin{smallmatrix} 1 &0 \\ 0 &-1 \end{smallmatrix} \right)$. Indeed, taking 
 \[ C:= \left(\begin{matrix} 1 &-\frac{1}{\sqrt{a}} \\ \sqrt{a} &1 \end{matrix} \right) \in \GL(2,K_a) \qquad \textrm{\edit{leads to}} \qquad \iota_{C^{-1}}\circ \theta_a \circ \iota_{C}= \iota_{C^{-1}A C}=\iota_B.\] 

It is easy to see \edit{that} the fixed point group in $\SL(2,K_a)$ of the involution $\iota_B$ is  $\Diag(2, K_a)$,  the full diagonal subgroup of $\SL(2,K_a)$.   
Thus, from above, the fixed point group in $\SL(2,K_a)$ of the involution $\iota_A$ is the subgroup $$\Fix_{\SL(2,K_a)}(\{\xi_{-},\xi_{+}\}) = C \Diag(2,K_a)C^{-1} \leq \SL(2,K_a)$$ which is the stabilizer in $\SL(2,K_a)$ of the bi-infinite geodesic line $[\xi_{-},\xi_{+}]$ in the Bruhat--Tits tree $T_{K_a} \cup \partial T_{K_a}$, and fixes pointwise the two ends $\xi_{-},\xi_{+}$. %In particular,   $\Fix_{\SL(2,K_a)}(\{\xi_{-},\xi_{+}\})$  is $\GL(2,K_a)$-conjugate to the diagonal subgroup of $\SL(2,K_a)$. 
Then $H_{\theta_a}=\Fix_{\SL(2,K_a)}(\{\xi_{-},\xi_{+}\}) \cap \SL(2,F)$. 

Notice, the geodesic line $[\xi_{-},\xi_{+}]$  appears in the Bruhat--Tits tree $T_F \cup \partial T_F$ if and only if $a=1$. The rest of the Corollary is an easy verification, and an application of Proposition \ref{prop::prop_shape_inv}.
\end{proof}

\medskip 
 In the next few pages we will compute the fixed point groups $H_\theta$ for abstract involutions $\theta= \iota_A \circ \sigma$ \editA{from Theorem \ref{thm::inv_SL_E}}. \edit{We will show: }

\begin{theorem}
\label{thm::inv_SL_E_2}
Let  $E=\QQ_p(\alpha)$ be a quadratic extension of $\QQ_p$. Consider any abstract involution $\theta = \iota_A \circ \sigma$ of $\SL(2,E)$ as in Theorem \ref{thm::inv_SL_E}, where $A \in \{\left( \begin{smallmatrix} z &y \\ 1 & -\sigma(z) \end{smallmatrix} \right),  \left( \begin{smallmatrix} x &0 \\ 0 & 1 \end{smallmatrix} \right)\} \subset \GL(2,E)$, with $y \in \QQ_p, z,x\in E$, with $z\sigma(z)+ y \neq 0$ and $x\sigma(x)=1$.
Then one of the following holds for $H_{\theta}:=\{g \in \SL(2,E) \; \vert \; \theta(g)= g\}$:
\begin{enumerate}
\item
$H_{\theta}$ is  $\GL(2,E)$-conjugate to $\SL(2,\QQ_p)$
\item
$H_{\theta}=B \SL(2,K^{\sigma}) B^{-1} \cap \SL(2,E)$, for some matrix  $B \in \GL(2,K)$, where $K \neq E$ is some finite field-extension of $E$ and $K^{\sigma}:=\{x \in K \; \vert \; \sigma(x)=x\}$ is the maximal subfield of $K$ \edit{fixed by $\sigma$}. In particular,  $H_\theta$ is compact or trivial.
\end{enumerate}
\end{theorem}

\begin{proof} Consider $\theta= \iota_A \circ \sigma$, with  $A=  \left( \begin{smallmatrix} x &0 \\ 0 & 1 \end{smallmatrix} \right)$ and where $x=x_1 +\edit{\alpha} x_2\in E=\QQ_p(\alpha)$ such that $x\sigma(x)=1$.  \editA{If $x=1$, then clearly $H_\theta = \SL(2, \QQ_p)$. If $x\neq 1$}, there are two cases: 
\[\begin{array} {cc} 
\textrm{Case 1: } \editA{\textrm{Suppose }x = - 1 \textrm{ then } x_1 = - 1, x_2 =0} 
%x = \pm 1 \Rightarrow x_1 = \pm 1, x_2 =0 
& \textrm{Case 2: } \editA{\textrm{Suppose } x \neq \pm 1  \textrm{ then } x_2 \neq 0} \\
%x \neq \pm 1  \Rightarrow x_2 \neq 0 \\
\editA{\textrm{and set }}B= \left( \begin{matrix} \alpha &0 \\ 0 & 1 \end{matrix} \right)  & \editA{\textrm{and set }} B= \left( \begin{matrix} \frac{x_1+1}{x_2} + \alpha &0 \\ 0 & 1 \end{matrix} \right) 
\end{array}.  \] 
%If $x=\pm 1$, then $x_1=\pm 1$ and $x_2=0$. Then take $B= \left( \begin{matrix} \alpha &0 \\ 0 & 1 \end{matrix} \right) \in \GL(2,E)$ when $x=-1$.  If $x\neq \pm 1$, then $x_2 \neq 0$, and just take $B= \left( \begin{matrix} \frac{x_1+1}{x_2} + \alpha &0 \\ 0 & 1 \end{matrix} \right) \in \GL(2,E)$. 
Then in both cases we see that  $\iota_{B^{-1}}\circ \theta \circ \iota_{B}= \iota_{B^{-1}A \sigma(B)}\circ \sigma=\sigma$,  implying that $\theta=\iota_A \circ \sigma$ is $\GL(2,E)$-conjugate to $\sigma$ via the map $\iota_{B^{-1}}$. Thus in this case $H_{\theta}:=\{g \in \SL(2,E) \; \vert \; \theta(g)= g\}$ is  $\GL(2,E)$-conjugate to $\SL(2,\QQ_p)$. 

\medskip 
Finally, consider  the general case \editA{from Theorem \ref{thm::inv_SL_E}} where $\theta= \iota_A \circ \sigma$, with  $A=  \left( \begin{smallmatrix} z &y \\ 1 & -\sigma(z) \end{smallmatrix} \right)$ and where $y \in \QQ_p, z\in E$ such that $z\sigma(z)+ y \neq 0$. Recall that 
\begin{equation}\label{c_eval} A\sigma(A)=c\Id, \text{ for some }c \in E^{*}. \end{equation}

Notice, the action of $\sigma$ on the Bruhat--Tits tree of $\SL(2,E)$ is again the natural map induced by $\sigma$, as the vertices of $T_E$ can be labeled with $a+\alpha b$, where $a,b\in \QQ_p$ (see \edit{F}igures \ref{fig::fig_unram}, \ref{fig::fig_ram}).  In other words, $\sigma$ induces an action on $T_E$ which fixes the subtree for $T_{\QQ_p}$ and acts as an involution on the remaining branches in $T_E - T_{\QQ_p}$. Then the involution $\theta_1$ of the Bruhat--Tits tree of $\SL(2,E)$ induced by the involution $\theta$ of $\SL(2,E)$ is just the map $\theta_1 = A\circ \sigma$. The same involution $\theta_1= A \circ \sigma$ acts on the boundary $\partial T_E$ of $T_E$.
Moreover, since $A \in \GL(2,E)$, the map $A$ acts on the $\partial T_E$ as an automorphism, and thus is a bijection on $\partial T_E$.

In order to classify the involutions $\theta$ up to $\GL(2,E)$-conjugacy, we \edit{first} want to \editA{find solutions} of the equation $\theta_1(\xi)=A(\sigma(\xi))=\xi$, with $\xi \in \partial T_E$, \editA{and use those solutions to find matrices $B \in \GL(2,E)$ with $A\sigma(B)=cB$ and $c \in E^{*}$}. If there is no such solution, we will solve \edit{this} equation for $\xi \in \partial T_K$, where $K$ is an appropriate finite extension of $E$. Notice that the actions of $A$ and $\sigma$ extend naturally to $T_K \cup \partial T_K$, with $\sigma$ fixing pointwise any primitive element in $K -E$. The action of $\sigma$ extends to $K$, since all the extensions are algebraic and we have $\QQ_p \subset E \subset K$.  

%then any element of $K$ may be written $x_1 + \alpha x_2$ where $x_1, x_2 \in K^{\sigma}$.

Putting together the above remarks, it is enough to find all the solutions of the equation $A(\sigma(A(\xi)))=A(\xi)$, for $\xi \in \partial T_K$ and $K$ a finite extension of \editB{$E$.} \editA{Since $A(\sigma(A(\xi)))=A\sigma(A)(\sigma(\xi))=c\sigma(\xi)$ where $ \xi \in P^{1}K= \big[ \begin{smallmatrix} 0 \\1  \end{smallmatrix} \big]   \bigsqcup\limits_{x\in K} \big[ \begin{smallmatrix} 1 \\x  \end{smallmatrix} \big]$ it} is then equivalent to finding all the $\editA{c'} \in K^{*}$ and $x \in K$ such that 
\begin{equation}
\label{equ::solve_ends}
 \begin{pmatrix} 1 \\ \sigma(x) \end{pmatrix}=\editA{c'}A\begin{pmatrix} 1 \\x \end{pmatrix}, \; \;  \textrm{or} \; \;  \begin{pmatrix} 0 \\ 1 \end{pmatrix}=\editA{c'}A  \begin{pmatrix} 0 \\ 1 \end{pmatrix}
\end{equation}
where $K$ is an appropriate finite extension of $E$ that will be determined below. 
First notice that any element $x \in K$ can be uniquely written as $x=x_1 + \alpha x_2$, with $x_1,x_2$ elements in $K$ such that $\sigma(x_i)=x_i$.

Thus, write  
\[ \left( \begin{matrix} z &y \\ 1 & -\sigma(z) \end{matrix} \right) = A = A_1 + \alpha A_2 =\left( \begin{matrix} z_1 &y \\ 1 & -z_1 \end{matrix} \right) + \alpha \left( \begin{matrix} z_2 &0 \\ 0 & z_2 \end{matrix} \right), \editB{\textrm{ and }} \xi =\bbmat x_1 + \alpha x_2 \\y_1 +\alpha y_2 \ebmat, \]   
and $A_1,A_2 \in M_2(\QQ_p)$, for $x_1,x_2,y_1,y_2 \in K$ such that $\sigma(x_i)=x_i$, $\sigma(y_i)=y_i$.  \editA{Since we will have more flexibility by considering $\xi$ in its general form $\bbmat x_1 + \alpha x_2 \\y_1 +\alpha y_2 \ebmat$,} we want to solve the equation for $\xi \in \partial T_K$ \editA{and $c\in K^{*}$ as in \eqref{c_eval}}:
\begin{equation}
\label{equ::equ_3}
c\begin{pmatrix} x_1 - \alpha x_2 \\y_1 -\alpha y_2 \end{pmatrix}=( A_1 + \alpha A_2)\begin{pmatrix} x_1 + \alpha x_2 \\y_1 +\alpha y_2 \end{pmatrix}
= A_1 \begin{pmatrix} x_1 \\y_1 \end{pmatrix} + \alpha^{2}A_2 \begin{pmatrix} x_2 \\ y_2 \end{pmatrix} +\alpha A_1\begin{pmatrix} x_2 \\ y_2 \end{pmatrix} + \alpha A_2  \begin{pmatrix} x_1 \\y_1 \end{pmatrix}
\end{equation}
where $c=c_1 + \alpha c_2 \in K^{*}$ and $\sigma(c_i)=c_i$. This equation (\ref{equ::equ_3}) is equivalent to the following system of equations:
\begin{equation}
\label{equ::solve_A_equ}
(A_1-c_1\Id) \begin{pmatrix} x_1 \\y_1 \end{pmatrix} = - \alpha^{2}(z_2+c_2) \begin{pmatrix} x_2 \\ y_2 \end{pmatrix},   \; \;  \; \;  (A_1 + c_1\Id) \begin{pmatrix} x_2 \\y_2 \end{pmatrix} = -(z_2-c_2)\begin{pmatrix} x_1 \\ y_1 \end{pmatrix}.
\end{equation}

\editA{Multiplying the second equality of \eqref{equ::solve_A_equ} by $A- c_1 \Id$ and using the first equality, we have: }
%Composing the second equality with the first:
\begin{equation}
\label{equ::solve_A_equ00}
(A_1^{2}-c_1^{2}\Id) \begin{pmatrix} x_2 \\y_2 \end{pmatrix}= (z_1^{2}+y-c_1^{2}) \begin{pmatrix} x_2 \\y_2 \end{pmatrix}= \alpha^{2}(z_2^{2}-c_2^{2})\begin{pmatrix} x_2 \\ y_2 \end{pmatrix}.
\end{equation}

 \edit{We solve Equation (\ref{equ::solve_A_equ00}) by breaking it into cases, by} comparing when the coefficients $(z_1^{2}+y-c_1^{2}) $ and $\alpha^{2}(z_2^{2}-c_2^{2})$ are equal to each other, or zero.

\textbf{Case 1}: If  $\editA{0 \neq }z_1^{2}+y-c_1^{2} \neq  \alpha^{2}(z_2^{2}-c_2^{2}) \neq 0$ then we must have $ \big( \begin{smallmatrix} x_2 \\ y_2 \end{smallmatrix} \big) =\big( \begin{smallmatrix} 0 \\ 0 \end{smallmatrix} \big)$, and so also  $\big( \begin{smallmatrix} x_1 \\ y_1 \end{smallmatrix} \big) =\big( \begin{smallmatrix} 0 \\ 0 \end{smallmatrix} \big)$. Thus there is no solution $\xi \in \partial T_K$ for the system (\ref{equ::solve_A_equ}), for any finite extension $K$.

\textbf{Case 2.1}: If  $z_1^{2}+y-c_1^{2} \neq 0$, $z_2^{2}-c_2^{2}=0$, and $z_2\neq 0 \neq c_2$. Then  $\big( \begin{smallmatrix} x_2 \\ y_2 \end{smallmatrix} \big)= \big(\begin{smallmatrix} 0 \\ 0 \end{smallmatrix}\big)$. If moreover, $z_2=-c_2$, from (\ref{equ::solve_A_equ}) we get that   $\big( \begin{smallmatrix} x_1 \\ y_1 \end{smallmatrix} \big)=\big( \begin{smallmatrix} 0 \\ 0 \end{smallmatrix} \big)$. If instead $z_2=c_2$, from (\ref{equ::solve_A_equ}) we get that  $(A_1-c_1\Id) \big( \begin{smallmatrix} x_1 \\y_1 \end{smallmatrix} \big) = \big( \begin{smallmatrix} 0 \\ 0 \end{smallmatrix} \big)$ and because $\det(A_1-c_1\Id)=-(z_1^2-c_1^2 +y) \neq 0$ we also get   $\big( \begin{smallmatrix} x_1 \\ y_1 \end{smallmatrix}\big) =\big( \begin{smallmatrix} 0 \\ 0 \end{smallmatrix} \big)$. Thus in this Case 2  there is no solution $\xi \in \partial T_K$ for the system (\ref{equ::solve_A_equ}).

\textbf{Case 2.2}: If  $z_1^{2}+y-c_1^{2} \neq 0$,  and $z_2= 0 =c_2$.   This means that  $A_2= \left( \begin{smallmatrix} 0 &0 \\ 0 & 0 \end{smallmatrix} \right)$, and so $A=A_1=\left( \begin{smallmatrix} z_1 &y \\ 1 & -z_1 \end{smallmatrix} \right) \in \GL(2,\QQ_p)$, with $A^{2}=(z_1^2+y)\Id$.  The system of equations (\ref{equ::solve_A_equ}) reduces to 
\begin{equation}
\label{equ::solve_A_equ1}
(A_1-c_1\Id) \begin{pmatrix} x_1 \\y_1 \end{pmatrix} = \begin{pmatrix} 0 \\ 0 \end{pmatrix} ,  \; \;  \; \;   (A_1 + c_1\Id) \begin{pmatrix} x_2 \\y_2 \end{pmatrix} =  \begin{pmatrix} 0 \\ 0 \end{pmatrix}.
\end{equation}
As the matrices $(A_1-c_1\Id)$ and $(A_1+c_1\Id)$ are invertible, with determinant $-(z_1^2-c_1^2 +y) \neq 0$, we get $ \big( \begin{smallmatrix} x_1 \\y_1 \end{smallmatrix} \big) = \big( \begin{smallmatrix} 0 \\ 0 \end{smallmatrix} \big) = \big( \begin{smallmatrix} x_2 \\y_2 \end{smallmatrix}\big)$. Again there is no solution $\xi \in \partial T_K$ for the system (\ref{equ::solve_A_equ}).

\textbf{Case 3}: If  $z_1^{2}+y-c_1^{2} = 0$ and $z_2^{2}-c_2^{2}\neq 0$. Then from (\ref{equ::solve_A_equ00}) we get  $\big( \begin{smallmatrix} x_2 \\y_2 \end{smallmatrix} \big)=\big( \begin{smallmatrix} 0 \\ 0 \end{smallmatrix} \big)$ and then from (\ref{equ::solve_A_equ})  $\big( \begin{smallmatrix} x_1 \\y_1 \end{smallmatrix}\big) =\big( \begin{smallmatrix} 0 \\ 0 \end{smallmatrix} \big)$.  Again there is no solution $\xi \in \partial T_K$ for the system (\ref{equ::solve_A_equ}).

\medskip
\edit{Thus the remaining cases to check for the coefficients that are not covered above is to solve for  $\xi \in \partial T_K$ when}:
\begin{enumerate} 
\item 
$0\neq z_1^{2}+y-c_1^{2} =\alpha^{2}(z_2^{2}-c_2^{2})\neq 0$, and where we take $K$ to be the minimal field extension of $\QQ_p$ that contains $c_1, c_2$
\item
$z_2^2 = c_2^2$, thus for $c_2 \in \QQ_p$, and  $z_1^{2}+y-c_1^{2} = 0$ for some $c_1 \in K$,  where $K$ is a minimal extension of $E$ in which $z_1^{2}+y$ is a square.
\end{enumerate}
 
 In each of the Cases 4, \edit{5.1, 5.2,} 5.3, one checks $B_i \in \GL(2,K)$ since $\det (B_i) \neq 0$ and also that $A\sigma(B_i)=(c_1+\alpha c_2)B_i=cB_i$.
 
\textbf{Case 4}:  If  $0\neq z_1^{2}+y-c_1^{2} =\alpha^{2}(z_2^{2}-c_2^{2})\neq 0$.  Then take the matrix 
$$B_1:=\left( \begin{matrix} \frac{y}{z_2-c_2}  &  \frac{z_1+c_1}{z_2-c_2}+ \alpha \\ \frac{-z_1+c_1}{z_2-c_2}+ \alpha & \frac{1}{z_2-c_2}  \end{matrix} \right).$$ 
%One can check that the determinant of $B_1$ equals $\frac{2\alpha^2 c_2}{z_2-c_2} - \alpha \frac{2 c_1}{z_2-c_2}$ which is not zero as $c=c_1 +\alpha c_2 \neq 0$. Thus $B_1 \in \GL(2,K)$ and one can check further that $A\sigma(B_1)=(c_1+\alpha c_2)B_1=cB_1$. 

%In particular, we have that  $\iota_{B_0^{-1}}\circ \theta \circ \iota_{B_0}= \iota_{B_0^{-1}A \sigma(B_0)}\circ \sigma=\sigma$, impying that $\theta=\iota_A \circ \sigma$ is $\GL(2,K)$-conjugated to $\sigma$, via the map $\iota_{B_0^{-1}}$. Thus in this case $H_{\theta}:=\{g \in \SL(2,E) \; \vert \; \theta(g)= g\}$ is $\GL(2,K)$-conjugated to $\SL(2,\QQ_p)$. 

\textbf{Case 5.1}: If  $z_1^{2}+y-c_1^{2} = 0$, $z_2^{2}-c_2^{2}= 0$, $y\neq 0$, and $z_2=c_2 \neq 0$. Then take the matrix 
$$B_2: = \left( \begin{matrix} -\frac{\alpha^2 (\edit{2z_2})(z_1 +c_1)}{y} +\alpha &\alpha^{2}(\edit{2z_2})+(z_1+c_1) + \alpha(z_1-c_1) \\ -\frac{\alpha(z_1+c_1)}{y} & 1+\alpha \end{matrix} \right).$$ 
%One can check that $B_2 \in \GL(2,K)$ and $A\sigma(B_2)=(c_1+\alpha c_2)B_2=cB_2$. 

%In particular, we have that  $\iota_{B_1^{-1}}\circ \theta \circ \iota_{B_1}= \iota_{B_1^{-1}A \sigma(B_1)}\circ \sigma=\sigma$, impying that $\theta=\iota_A \circ \sigma$ is $\GL(2,E)$-conjugated to $\sigma$, via the map $\iota_{B_1^{-1}}$. Thus in this case $H_{\theta}:=\{g \in \SL(2,E) \; \vert \; \theta(g)= g\}$ is $\GL(2,E)$-conjugated to $\SL(2,\QQ_p)$. 

\textbf{Case 5.2}: If  $z_1^{2}+y-c_1^{2} = 0$, $z_2^{2}-c_2^{2}= 0$, $y\neq 0$, and $z_2=-c_2 \neq 0$. Then take the matrix 
$$\editA{B_3: = \left( \begin{matrix} 1 +\alpha\frac{2z_2}{z_1+c_1}  & (z_1+c_1) -\alpha(z_1-c_1-2z_2) \\ \frac{1}{z_1+c_1} & 1-\alpha \end{matrix} \right).}$$ 
%One can check that $B_3 \in \GL(2,K)$ and $A\sigma(B_3)=(c_1+\alpha c_2)B_3=cB_3$. 

%In particular, we have that  $\iota_{B_2^{-1}}\circ \theta \circ \iota_{B_2}= \iota_{B_2^{-1}A \sigma(B_2)}\circ \sigma=\sigma$, impying that $\theta=\iota_A \circ \sigma$ is $\GL(2,E)$-conjugated to $\sigma$, via the map $\iota_{B_2^{-1}}$. Thus in this case $H_{\theta}:=\{g \in \SL(2,E) \; \vert \; \theta(g)= g\}$ is again $\GL(2,E)$-conjugated to $\SL(2,\QQ_p)$. 

\textbf{Case 5.3}: If  $z_1^{2}+y-c_1^{2} = 0$, $z_2=c_2= 0$, and $y\neq 0$.   This means that $A_2= \left( \begin{smallmatrix} 0 &0 \\ 0 & 0 \end{smallmatrix} \right)$, and so  $A=A_1=\left( \begin{smallmatrix} z_1 &y \\ 1 & -z_1 \end{smallmatrix} \right) \in \GL(2,\QQ_p)$.
Then take the matrix 
$$B_4: = \left( \begin{matrix} \alpha (z_1-c_1)  & 1 \\ \alpha & \frac{c_1-z_1}{y} \end{matrix} \right).$$
%One can check that $B_4 \in \GL(2,K)$ and $A\sigma(B_4)=c_1B_4$. 

%In particular, we have that  $\iota_{B_3^{-1}}\circ \theta \circ \iota_{B_3}= \iota_{B_3^{-1}A \sigma(B_3)}\circ \sigma=\sigma$, impying that $\theta=\iota_A \circ \sigma$ is $\GL(2,E)$-conjugated to $\sigma$, via the map $\iota_{B_3^{-1}}$. Thus in this case $H_{\theta}:=\{g \in \SL(2,E) \; \vert \; \theta(g)= g\}$ is again $\GL(2,E)$-conjugated to $\SL(2,\QQ_p)$. 

 \textbf{Case 5.4}: If  $z_1^{2}-c_1^{2} = 0$, $z_2^2 -c_2^2= 0$, and $y= 0$.  This means that $c=c_1 +\alpha c_2 \in E$ and  $A=\left( \begin{smallmatrix} z &0 \\ 1 & -\sigma(z) \end{smallmatrix} \right) \in \GL(2,E)$. Since multiplication by a constant in $E^{*}$ will give the same results, we can multiply $A$ with $-\frac{1}{\sigma(z)}$, and we get a matrix $A= \left( \begin{smallmatrix} x &0 \\ w & 1 \end{smallmatrix} \right)$ such that $x\sigma(x)=1$ and $w=-x\sigma(w)$, with $x,w \in E^*$.

By our assumption $z_1 \neq 0$ or $z_2 \neq 0$.  There are a few final subcases, \edit{set $x=x'_1 +\alpha x'_2$}, and $w=w_1 +\alpha w_2$: 

\[ \begin{array}{ccc}  \editA{\textrm{if }}z_1 z_2 \neq 0  & \editA{\textrm{then }}  \edit{x' _2} \neq 0 , w_2 \neq 0 & 
\editA{\textrm{and take }}B_5:= \left( \begin{matrix} \frac{1+\edit{x'_1}}{\edit{x'_2}} + \alpha &0 \\  \alpha \frac{w_2}{\edit{x'_2}}  & 1 \end{matrix} \right) \\ 
&&\\
 \editA{\textrm{if }}z_1 z_2 =0 & \textrm{and } x =  1,  \editA{\textrm{ then }} w =\alpha w_2 & \editA{\textrm{take }}B_5:=  \left( \begin{smallmatrix} 1& 0 \\  \alpha \frac{w_2}{2} &1 \end{smallmatrix} \right) \\
&&\\
 \editA{\textrm{if }}z_1 z_2 =0 & \textrm{and } x = -1, \editA{\textrm{ then }} w= w_1 & \editA{\textrm{take }}B_5:=  \left( \begin{smallmatrix}  \alpha &  \alpha  \\  1- \alpha \frac{w_1}{2} &-1- \alpha \frac{w_1}{2}  \end{smallmatrix} \right).
\end{array} 
\] 

%If we take $z_1 z_2 \neq 0$, then we get that $x=x_1 +\alpha x_2$ with $x_2 \neq 0$, and $w=w_1 +\alpha w_2$, with $w_2 \neq 0$. Then take 
%$$B_5:= \left( \begin{matrix} \frac{1+x_1}{x_2} + \alpha &0 \\  \alpha \frac{w_2}{x_2}  & 1 \end{matrix} \right) \in \GL(2,E).$$

%If $z_1z_2=0$, then $x=\pm1$ and $w =\mp \sigma(w)$. Then take 
%$$B_5:=  \left( \begin{smallmatrix} 1& 0 \\  \alpha \frac{w_2}{2} &1 \end{smallmatrix} \right) \in \GL(2,E)$$ 
%if $x = 1$, thus implying $w=\alpha w_2$. Or,  if $x = -1$ thus implying that $w=w_1$, take 
%$$B_5:=  \left( \begin{smallmatrix}  \alpha &  \alpha  \\  1- \alpha \frac{w_1}{2} &-1- \alpha \frac{w_1}{2}  \end{smallmatrix} \right) \in \GL(2,E).$$
In all those subcases we see \editA{$A\sigma(B_5)=B_5$ and note $\det(B_5) \neq 0$ in all cases. }

% We have $\left( \begin{smallmatrix} 1 &0 \\  \alpha w_2 &1 \end{smallmatrix} \right) \left( \begin{smallmatrix} 1& 0 \\ - \alpha \frac{w_2}{2} &1 \end{smallmatrix} \right) = \left( \begin{smallmatrix} 1& 0 \\  \alpha \frac{w_2}{2} &1 \end{smallmatrix} \right)$. 

%We have $\left( \begin{smallmatrix} -1 & 0 \\ w_1 &1 \end{smallmatrix} \right) \left( \begin{smallmatrix} -\alpha  &  - \alpha  \\ 1+ \alpha \frac{w_1}{2} & -1+ \alpha \frac{w_1}{2} \end{smallmatrix} \right) = \left( \begin{smallmatrix}  \alpha &  \alpha  \\  1- \alpha \frac{w_1}{2} &-1- \alpha \frac{w_1}{2}  \end{smallmatrix} \right)$.

\medskip
In all the above cases we  verified $\iota_{B_i^{-1}}\circ \theta \circ \iota_{B_i}= \iota_{B_i^{-1}A \sigma(B_i)}\circ \sigma=\sigma$, \edit{implying} that $\theta=\iota_A \circ \sigma$ is $\GL(2,K)$-conjugated to $\sigma$ via the map $\iota_{B_i^{-1}}$ (both viewed as abstract involutions of $\SL(2,K)$), where $K$ is $E$ or a finite field-extension of $E$.

If  the matrix $B_i$ is in $\GL(2,E)$, then the fixed point group  associated with $\theta$ is the group   $$ H_{\theta} :=\{g \in \SL(2,E) \; \vert \; \theta(g)= g\}= B_i \SL(2,\QQ_p) B_i^{-1}$$ 
that is  $\GL(2,E)$-conjugate to $\SL(2,\QQ_p)$.

But if the matrix $B_i$ is in $\GL(2,K)$, where $K \neq E$ is a finite field-extension of $E$, then the fixed point group  associated with $\theta$ viewed as an involution of $\SL(2,K)$,  is the group  $$H_{\theta}(K):=\{g \in \SL(2,K) \; \vert \; \theta(g)= g\}= B_i \SL(2,K^{\sigma}) B_i^{-1}$$
where $K^{\sigma}:=\{x \in K \; \vert \; \sigma(x)=x\}$ is the maximal subfield of $K$ which does not contain %the element 
 $\alpha$. 
Then 
$$H_{\theta}:=\{g \in \SL(2,E) \; \vert \; \theta(g)= g\}= B_i \SL(2,K^{\sigma}) B_i^{-1} \cap \SL(2,E).$$  \end{proof} 

\begin{lemma}
\label{lem::trivial_boundary_inv_B}
Consider one of the matrices $B_i \in \GL(2,K)$computed above, where $K \neq E$ is a finite field-extension of $E$. Then $B_i (\partial T_{K^{\sigma}}) \cap \partial T_E = \{\emptyset \}$ and $$H_\theta =B_i \SL(2,K^{\sigma}) B_i^{-1} \cap \SL(2,E)$$ is compact or trivial.
\end{lemma}
\begin{proof}
Take a matrix $B_i$ as above such that $c=c_1+\alpha c_2 \in K-E$, i.e. as in Cases 4, 5.1, 5.2, 5.3. Then the first observation is that the set of all ends in $\partial T_K$ that are pointwise fixed by the involution $\theta_1=A \circ \sigma$ on $T_K \cup \partial T_K$ induced from the involution $\theta=\iota_A \circ \sigma$, is the set $B_i(\partial T_{K^{\sigma}})$. In particular, for a representative  $\left( \begin{smallmatrix} x_1 + \alpha x_2 \\y_1 +\alpha y_2 \end{smallmatrix}\right)$ of such an end  $\xi =\big [ \begin{smallmatrix} x_1 + \alpha x_2 \\y_1 +\alpha y_2 \end{smallmatrix} \big] \in B_i(\partial T_{K^{\sigma}})$, where $x_1,x_2,y_1,y_2 \in K$ with $\sigma(x_i)=x_i$, $\sigma(y_i)=y_i$, we have that $A(\sigma(\xi))=\xi$ and this becomes
$$\theta_1\left(\begin{matrix} x_1 + \alpha x_2 \\y_1 +\alpha y_2 \end{matrix}\right)= A\left(\sigma\left(\begin{matrix} x_1 + \alpha x_2 \\y_1 +\alpha y_2 \end{matrix}\right)\right)= A\left(\begin{matrix} x_1 - \alpha x_2 \\y_1 -\alpha y_2 \end{matrix}\right)=c\left(\begin{matrix} x_1 + \alpha x_2 \\y_1 +\alpha y_2 \end{matrix}\right)$$
\edit{for the $c \in K-E$ considered above.}
Notice, our equations (\ref{equ::equ_3}) and (\ref{equ::solve_A_equ}) above are derived from the equation $A(\eta)=c\sigma(\eta)$, where $\xi=A(\eta)$.

Suppose that there is $\xi \in \partial T_E$ such that  $A(\sigma(\xi))=\xi$. Then this means that for a representative  $\left( \begin{smallmatrix} x_1 + \alpha x_2 \\y_1 +\alpha y_2 \end{smallmatrix}\right)$ of $\xi$, with $x_1,x_2,y_1,y_2 \in \QQ_p$ we must have 
\begin{equation} 
\label{equ::ends_K}
(A_1-c_1\Id) \begin{pmatrix} x_1 \\y_1 \end{pmatrix} = \alpha^{2}(z_2+c_2) \begin{pmatrix} x_2 \\ y_2 \end{pmatrix},   \; \;  \; \;  (A_1 + c_1\Id) \begin{pmatrix} x_2 \\y_2 \end{pmatrix} = (z_2-c_2)\begin{pmatrix} x_1 \\ y_1 \end{pmatrix}.
\end{equation}
 For Cases 5.1, 5.2, 5.3, a solution $\xi \in \partial T_E$ of the system (\ref{equ::ends_K}) will imply $c_1 \in \QQ_p$, and since in those cases $c_2 \in \QQ_p$,  in contradiction with the assumption that $c \in K -E$.

\edit{For} Case 4, taking separately the two equalities of \edit{System} (\ref{equ::ends_K}), we will get that 
$$c_1x_1 + c_2x_2\alpha^2, c_1y_1+c_2y_2 \alpha^2, c_1x_2 + c_2x_1, c_1y_2+c_2y_1 \in \QQ_p.$$
By multiplying accordingly, we will get that $c_2(x_1^2 -\alpha^2 x_2^2) \in \QQ_p$ and  $c_1(y_1^2 -\alpha^2 y_2^2) \in \QQ_p$, thus $c_1,c_2$ must be in $\QQ_p$, giving again a contradiction with our assumption.

By the first part of the lemma we now know that $B_i (\partial T_{K^{\sigma}}) \cap \partial T_E = \{\emptyset \}$. This implies that the intersection of the tree $B_i(T_{K^{\sigma}})$ with the tree $T_E$ is either empty or a finite subtree of $T_E$. In both such cases, we must have that $H_{\theta}$ stabilizes setwise a finite connected subtree of $T_K$, implying that $H_\theta$ is compact or trivial.
\end{proof}

%{\red Some additional questions: 1. Does $H_{\theta}$ have a finite number of orbits on the boundary $\partial T_E$? 2. Given $E/F$ a finite extension of $F$, how many closed/open orbits does $\SL(2,F)$ have on the boundary $\partial T_E$? }

\begin{remark}  We leave it as an open question to compute the possibilities for compact $H_\theta$ that preserve a finite subtree. 
\end{remark} 
%However, notice that this problem is deeper than it appears at first glance.  Pink \cite{Pink} Corollary 0.5, tells us that such subgroups have a filtration.   If in addition the compact subgroups are open (which must not be the case), then Bruhat--Tits theory tells us they are finite index.   This then becomes a matter of answering the congruence subgroup problem, see \cite{Raghunathan} and others, for $\SL(2, \ZZ_p)$.    We would also like to understand the action of the congruence subgroups on the tree, and we conjecture that the index of the subgroup is related to the degree of the extension $K$. 

\section{Polar decomposition of $\SL(2,E)$ with respect to various subgroups}\label{sec:polar}%%%%%%%%%%%%%%%%%%

Let $F$ be a finite field-extension of $\QQ_p$ and $E=F(\alpha)$ be a quadratic extension of $F$. Denote by $\omega$ a uniformizer of $F$. It is well known that $\vert E^{*}/(E^{*})^{2}\vert =4=\vert F^{*}/(F^{*})^{2}\vert$.  Recall, throughout this article we consider $p\neq 2$.

Inspired by techniques from representation theory and spherical varieties, we first give an upper bound for the number of orbits of various subgroups of $\SL(2,K)$ on the boundary $\partial T_K$, where $K$ is either $F$ or $E$. Those results will be used to apply the various polar decompositions proven in Proposition \ref{prop::polar_decom} and to compute Chabauty limits of groups of involutions, and also to provide different, direct, and more geometric proofs than in \cite{HW_class,HelWD,Hel_k_invol} for the case of $\SL_2$.

\begin{lemma}
\label{lem::nr_orbits_SL(F)_SL(E)}
There are at most $5$ $\SL(2,F)$-orbits on the boundary $\partial T_{E}\cong P^1E=E \cup \{\infty\}$: 
\begin{enumerate}
\item 
the $\SL(2,F)$-orbit of  $\big [ \begin{smallmatrix} 1 \\ 0 \end{smallmatrix} \big]$
\item
the $\SL(2,F)$-orbits of  $\big[ \begin{smallmatrix} 1 \\ m\alpha \end{smallmatrix} \big]$, for each $m \in F^{*}/(F^{*})^{2}$,  that might coincide for different $m$'s.
\end{enumerate}
\end{lemma}
\begin{proof}
Indeed, consider the subgroup  $B:=\{ \left( \begin{smallmatrix} d^{-1} &0 \\ c & d\end{smallmatrix} \right) \; \vert \; c \in F, \; d \in F^{*}\}$ in $\SL(2,F)$.  The image of $\big[\begin{smallmatrix}  1 \\ m\alpha \end{smallmatrix} \big] $ under $B$, is \editA{the set $ \left \{ \big[\begin{smallmatrix} 1 \\ cd + d^{2}m\alpha \end{smallmatrix} \big]\;| \;c \in F, d \in F^* \right \} $ and these cover the entire $\partial T_{E}- \partial T_F$ since we may choose $c,d$ so that $cd + d^{2}m\alpha$ takes any value in $E$.} 
\end{proof}

\begin{lemma}
\label{lem::nr_orbits_diag_SL}
Let $K$ be a finite field-extension of $\QQ_p$. There are exactly $6$ orbits of the diagonal subgroup  $\edit{\Diag(2,K)}:=\{ \left( \begin{smallmatrix} d^{-1} &0 \\ 0 & d\end{smallmatrix} \right) \; \vert \; d \in K^{*}\} \leq \SL(2,K)$ on the boundary $\partial T_K$:
\begin{enumerate}
\item 
the $\edit{\Diag(2,K)}$-orbit of  $\big[\begin{smallmatrix} 1 \\ 0 \end{smallmatrix} \big]$ and the $\editB{\Diag(2,K)}$-orbit of  \big[$\begin{smallmatrix} 0 \\ 1 \end{smallmatrix} \big] $
\item
the $\edit{\Diag(2,K)}$-orbits of  $\big[\begin{smallmatrix} 1 \\ m \end{smallmatrix} \big]$, for each $m \in K^{*}/(K^{*})^{2}$. 
\end{enumerate}
\end{lemma}
\begin{proof}
The subgroup $\edit{\Diag(2,K)}$ fixes pointwise the ends  $\big[ \begin{smallmatrix} 1 \\ 0 \end{smallmatrix} \big]$ and $\big[\begin{smallmatrix} 0 \\ 1 \end{smallmatrix} \big]$.   \edit{For each  $m \in K^{*}/(K^{*})^{2}$, the $\edit{\Diag(2,K)}$-orbit of $\big[ \begin{smallmatrix}  1 \\ m \end{smallmatrix} \big]$,  is the set of vectors} of the form $\big[ \begin{smallmatrix} 1 \\ d^{2}m \end{smallmatrix} \big]$, \edit{and these vectors} cover the entire boundary   $\partial T_{E} - \{\big[ \begin{smallmatrix}  0 \\ 1 \end{smallmatrix} \big] ,\big[  \begin{smallmatrix} 1 \\ 0 \end{smallmatrix} \big]\}$.
\end{proof}

\begin{lemma}
\label{lem::nr_orbits_H_SL}
Let $F$ be a finite field-extension of $\QQ_p$, $A=\left( \begin{smallmatrix} 0 &1 \\ a &0 \end{smallmatrix} \right)$, with $a\in F^{*}/(F^{*})^{2}$, $a\neq 1$ and $\theta_a:=\iota_A$ the corresponding $F$-involution of $\SL(2,F)$.  
Then $H_{\theta_a}:=\{g \in \SL(2,F) \; \vert \; \theta_a(g) = g\}$ has at most $8$ orbits \edit{in} the boundary $\partial T_F$.
\end{lemma}
\begin{proof}
Take $K_a:=F(\sqrt{a})$ to be the quadratic field extension of $F$ corresponding to $a$. By Corollary \ref{lem::geom_k_inv_SL_2} we know that $H_{\theta_a}$ fixes pointwise the ends  $\xi_{\pm} := \big[ \begin{smallmatrix} 1 \\ \pm \sqrt{a} \end{smallmatrix} \big] $ in the boundary $\partial T_{K_a} - \partial T_F$, and $H_{\theta_a} = \Fix_{\SL(2,K_a)}(\{\xi_{-},\xi_{+}\}) \cap \SL(2,F)$. 

 The subgroup $\Fix_{\SL(2,K_a)}(\{\xi_{-},\xi_{+}\})$ is  $\SL(2,K_a)$-conjugate to $\Diag(2,K_a)$ \editA{where $C= \left(\begin{smallmatrix} 1 &-\frac{1}{\sqrt{a}} \\ \sqrt{a} &1 \end{smallmatrix} \right)$:}
\editA{$$\Fix_{\SL(2,K_a)}(\{\xi_{-},\xi_{+}\}) = C \Diag(2,K_a)C^{-1}=\edit{\left\{\frac{1}{2} \left(\begin{matrix} b+\frac{1}{b} &\frac{1}{\sqrt{a}}(b-\frac{1}{b}) \\ \sqrt{a} (b-\frac{1}{b})&b+\frac{1}{b} \end{matrix} \right)\; \vert \; b \in K_a^{*} \right\}}.$$}
\editA{By applying Lemma \ref{lem::nr_orbits_diag_SL}, (2), and since $|K_a^{*}/(K_a^{*})^{2}|=4$} we know that there are at most $4$ orbits of $\Fix_{\SL(2,K_a)}(\{\xi_{-},\xi_{+}\})$ \edit{in} the boundary  $\partial T_F \subset \partial T_{K_a} - \{\big [ \begin{smallmatrix} 1 \\ \pm \sqrt{a} \end{smallmatrix} \big] \}$.  Note the only elements in $\Diag(2,K_a)$ fixing $\xi \in  \partial T_{K_a} - \{\big[ \begin{smallmatrix} 1 \\ 0 \end{smallmatrix} \big], \big[ \begin{smallmatrix} 0 \\ 1 \end{smallmatrix} \big] \}$ are $\pm \Id.$
\edit{Notice} an element $h \in \Fix_{\SL(2,K_a)}(\{\xi_{-},\xi_{+}\})$ fixes an end  $\xi \in \partial T_{K_a} - \{\big[ \begin{smallmatrix}  1 \\ \pm \sqrt{a} \end{smallmatrix} \big] \}$ if and only if $h = \pm \Id$. %Indeed, this follows from the fact that  

Suppose now that there is $h \in \Fix_{\SL(2,K_a)}(\{\xi_{-},\xi_{+}\}) - H_{\theta_a}$, and $\xi_1 \neq \xi_2 \in \partial T_F$ such that $h(\xi_1)=\xi_2$. Then by applying $\theta_a$ to the latter equality \editA{and using $\theta_a (h(\xi))= h( \theta_a(\xi))$}, we also get $h(\theta_a(\xi_1))=\theta_a(\xi_2)$.  Because the only ends in $\partial T_{K_a}$ fixed by $\theta_a$ are  $\big[ \begin{smallmatrix}  1 \\ \pm \sqrt{a} \end{smallmatrix} \big]$ we have that $\xi_1 \neq \theta_a(\xi_1)$ and $\xi_2 \neq  \theta_a(\xi_2)$. Consider \edit{some} representatives $x_1,x_2$ of $\xi_1,\xi_2$ in $F^2$, respectively. Then there is some $w = w_1+ \sqrt{a} w_2 \in K_a^{*}$, with $w_1,w_2 \in F$, such that $h(x_1)=wx_2$, and so $h(\theta_{a}(x_1))=w\theta_{a}(x_2)$. \edit{Both matrices} $(x_1,\theta_a(x_1))$ and $(wx_2, w\theta_a(x_2))$ are invertible, the first with entries in $F$. Then $h (x_1,\theta_a(x_1))= w(x_2,\theta_a(x_2))$, and so 
$$h =w(x_2,\theta_a(x_2))(x_1,\theta_a(x_1))^{-1} \in \SL(2,K_a)$$ 
with $(x_2,\theta_a(x_2)) (x_1,\theta_a(x_1))^{-1} \in \GL(2,F)$. 
By taking the determinant in the latter equality, we have $w^2= w_1^{2}+aw_2^{2} + 2w_1w_2\sqrt{a} \in F^{*}$. This implies that either $w_2 =0$ or $w_1=0$.  If $w_2=0$ then $h \in \SL(2,F)$, which is a contradiction with our assumption that $h \notin H_{\theta_a}$. If $w_1=0$, then $w= w_2 \sqrt{a}$, \editA{and since $h \in \sqrt a \GL(2,F)$ then}
$$h \in \sqrt{a} \GL(2,F) \cap C \Diag(2,K_a)C^{-1}.$$

From the matrix form of $C \Diag(2,K_a)C^{-1}$, we have that  $h= \sqrt{a} \left(\begin{smallmatrix} b &c \\ ac&b \end{smallmatrix} \right)$, with $b,c \in F$ such that $b^2 -ac^2 = \frac{1}{a}$. If there exists such a solution \editA{$b+\sqrt{a}c$, with $b,c\in F$,} to the latter equation, then any other solution is of the form $(b+ \sqrt{a}c)(x+\sqrt{a}y)$, with $x,y \in F$ such that $x^2-a y^2 = 1$. \editA{Indeed, if $b_1^2 -ac_1^2 = \frac{1}{a}=b_2^2 -ac_2^2 $, then $(b_1+ \sqrt{a}c_1)(b_2+ \sqrt{a}c_2)^{-1}=:(x+ \sqrt{a}y)$ \editB{with} $x^2-a y^2 = 1$.} Thus \editA{by Corollary \ref{lem::geom_k_inv_SL_2}}, $h \in \sqrt{a} \left(\begin{smallmatrix} b &c \\ ac&b \end{smallmatrix} \right) \edit{H_{\theta_a}}$, with $b^2 -ac^2 = \frac{1}{a}$.

We have proved that the only elements $h \in \Fix_{\SL(2,K_a)}(\{\xi_{-},\xi_{+}\})$ that can  act nontrivially on ends of $\partial T_F$ are \edit{in} $H_{\theta_a}$ or belong to the coset  $\sqrt{a} \left(\begin{smallmatrix} b &c \\ ac&b \end{smallmatrix} \right) \edit{H_{\theta_a}}$, with $b^2 -ac^2 = \frac{1}{a}$. This implies  there are at most $8$ $H_{\theta_a}$-orbits in $\partial T_F$. This proves the lemma.
\end{proof}

Let  $E=\QQ_p(\alpha)$ be a quadratic extension of $\QQ_p$. Consider any abstract involution $\theta = \iota_A \circ \sigma$ of $\SL(2,E)$ as in Theorem \ref{thm::inv_SL_E}, where $A \in \{\left( \begin{smallmatrix} z &y \\ 1 & -\sigma(z) \end{smallmatrix} \right),  \left( \begin{smallmatrix} x &0 \\ 0 & 1 \end{smallmatrix} \right)\} \subset \GL(2,E)$, with $y \in \QQ_p, z,x\in E$, with $z\sigma(z)+ y \neq 0$ and $x\sigma(x)=1$. Let  $H_{\theta}:=\{g \in \SL(2,E) \; \vert \; \theta(g) = g\}$ and suppose $H_{\theta}=B \SL(2,K^{\sigma}) B^{-1} \cap \SL(2,E)$, for some matrix  $B \in \GL(2,K)$, where $K \neq E$ is some finite field-extension of $E$ and $K^{\sigma}:=\{x \in K \; \vert \; \sigma(x)=x\}$ is the maximal subfield of $K$ with that property.

Are there infinitely or finitely many $H_{\theta}$-orbits on the boundary $\partial T_E$? We give a partial answer in the following remark.

\begin{remark} There may be an infinite number of $H_{\theta}$-orbits on the boundary $\partial T_E$.  To see this, note:
By Lemma \ref{lem::trivial_boundary_inv_B} we know that $B (\partial T_{K^{\sigma}}) \cap \partial T_E = \{\emptyset \}$. Notice that $K$ is a quadratic extension of $K^{\sigma}$, i.e. $K=K^{\sigma}(\alpha)$. By Lemma \ref{lem::nr_orbits_SL(F)_SL(E)} there are at most $5$ $\SL(2,K^{\sigma})$-orbits on the boundary $\partial T_K$. By conjugating the groups $\SL(2,K), \SL(2,K^{\sigma})$ with the matrix $B\in \GL(2,K) $, one easily deduces  there are at most $5$ $B\SL(2,K^{\sigma})B^{-1}$-orbits on the boundary $\partial T_K$.

If we use the same computational trick as in Lemma \ref{lem::nr_orbits_H_SL} one can see that the number of $H_{\theta}$-orbits on the boundary $\partial T_E$ might  be infinite. Indeed, suppose there is $h \in B\SL(2,K^{\sigma})B^{-1} - H_{\theta}$ and $\xi_1,\xi_2 \in \partial T_E$ such that $h(\xi_1)=\xi_2$. By taking representatives $x_1,x_2$ of $\xi_1,\xi_2$ in $E^2$, respectively, we will have that $h(x_1)=wx_2$, for some $w \in K$. Then apply the involution $\theta$ and get $h(\theta(x_1))= \sigma(w)\theta(x_2)$. Since $B (\partial T_{K^{\sigma}}) \cap \partial T_E = \{\emptyset \}$, we know that $\theta(\xi_i)\neq \xi_i$ and so the the matrices $(x_i,\theta(x_i))$ are invertible and in $\GL(2,E)$. So $h = (x_2,\theta(x_2)) \left( \begin{smallmatrix} w &0 \\ 0 & \sigma(w) \end{smallmatrix} \right)  (x_1,\theta(x_1))^{-1} \in \SL(2,K)$ and by taking the determinant we must have $w\sigma(w) \in E$ and thus  $w\sigma(w) \in \QQ_p^{*}$, where $w \in K$. By considering squares in $\QQ_p^{*}$, we restrict to the case when $w\sigma(w) \in \QQ_p^{*}/( \QQ_p^{*})^{2}$ with $w \in K$. The latter has infinitely many solutions over $K$ and so there might be an infinite number of left $H_{\theta}$-cosets in $ B\SL(2,K^{\sigma})B^{-1}$, implying a possibly infinite number of $H_{\theta}$-orbits in $\partial T_E$.
\end{remark}

\medskip
To prove our results we need a polar decomposition of $\SL(2,E)$ with respect to $\SL(2,F)$. We cannot apply directly \cite{BenoistOh} as their result is proven only for $k$-involutions and not for abstract involutions. We will keep the notation $\theta_a$ \editA{for inner conjugation by $\left (\begin{smallmatrix}0 & 1 \\ a& 0 \end{smallmatrix} \right)$} from Corollary \ref{lem::geom_k_inv_SL_2}, and $\theta$ will refer to the involutions from Theorem \ref{thm::inv_SL_E_2}(2). Here we give a general result  which covers $k$-involutions as well as abstract involutions for  $\SL(2,k)$ where $k$ is a non-Archimedean local field of characteristic 0. \editA{Recall that intuitivley we want to decompose $G$ as the product of $H$ with a set `perpendicular' to $H$ with respect to the corresponding involution.}

\begin{proposition}[The polar decomposition for various subgroups of $\SL_2$]
\label{prop::polar_decom}
Let $F$ be a finite field-extension of $\QQ_p$, $E=F(\alpha)$ be a quadratic extension of $F$, \editA{and $a\in F^{*}/(F^{*})^{2}$, $a\neq 1$}. Let $H \leq G$ be \edit{one of the pairs} 
$$(G,H) \in \{(\SL(2,F),H_{\theta_a}), (\SL(2,E),\SL(2,F)), (\SL(2,\QQ_p(\alpha)),H_{\theta})\}.$$
Denote by $T_G$ the  Bruhat--Tits tree of $G$ and by $T_H$ the (possibly finite)  $H$-invariant subtree of $T_G$. Let $I_H$ be the number of $H$-orbits in the ideal boundary $\partial T_G -\partial T_H$, and $\xi_i$ a representative in each such orbit. Let $x_0 \in T_G \cap T_H$ be a vertex, which for the pair $ (\SL(2,E),\SL(2,F))$ will be taken to be the point $0$ as in the \edit{Figures} \ref{fig::fig_unram}, \ref{fig::fig_ram}.
Then 
$$G= \mathcal{K}\mathcal{B}H$$ where $\mathcal{K}$ is a compact subset of $G$ \editA{that depends on $x_0$}, and $\mathcal{B}= \{\Id\} \bigsqcup\limits_{i\in I_H} A_i$, where $A_i:= \{a_i^{n} \; \vert \; n \in \ZZ\}$ with $a_i \in G$ a hyperbolic element of translation length $2$ and with attractive endpoint in the $H$-orbit of $\xi_i$.
\end{proposition}
\begin{proof}
Let us make first some useful remarks. As the group $\SL(2)$ over a non-Archimedean local field acts by type-preserving automorphisms, thus edge-transitively, on its Bruhat--Tits tree, our groups $G,\editA{H=\SL(2,F)}$ will do the same on $T_G$, $T_H$, respectively. Denote by $V$ a fundamental domain of $H$ acting on $T_H$, \editA{which contains the vertex, $x_0 \in V$.} If $H= H_{\theta}$, or $H= H_{\theta_a}$, with $a \neq 1$, then $T_H$ is a finite subtree, and we can just consider $V=T_H$. In those cases the ideal boundary $\partial T_H$ is $\{\emptyset\}$. If $H=\SL(2,F)$ or $H= H_{\theta_a}$, with $a=1$, then $T_H$ is the Bruhat--Tits tree of $\SL(2,F)$, or a bi-infinite geodesic line in $T_G$, respectively. In those two cases the fundamental domain $V$ is an edge in $T_H$. Moreover, \editA{for those two cases,} the boundary $\partial T_H$ is the projective space $P^{1}F$, and  two endpoints of $\partial T_E$, respectively. Notice that for the case when  $G= \SL(2,E)$ and $H=\SL(2,F)$, the edge $V$ is either an edge of $T_E$ (for $E$ unramified), or the union of two consecutive edges of $T_E$ (for $E$ ramified).  

For each of the $H$-orbits in $\partial T_G - \partial T_H$, thus for each $i \in I_H$, we can choose a representative $\xi_i \in \partial T_G - \partial T_H$ such that its projection  $x_i$ on the tree $T_H$ is in the fundamental domain $V$. Then  $x_i \in V$ (which is viewed as a subset of $T_G$), but \edit{is} not \edit{necessarily} a vertex (e.g. this is the case for $G= \SL(2,E), H=\SL(2,F)$ and  $E=F(\alpha)$ a ramified extension of $F$). 

This means that the geodesic ray $[x_i,\xi_i) \subset T_G$ that starts from $x_i \in  V$ and  with endpoint $\xi_i$ is entirely disjoint from the tree $T_{H}$, except its basepoint $x_i$. Denote by $a_i$ a hyperbolic element of $G$ with translation length $2$, translation axis containing the geodesic ray $[x_i, \xi_i)$, and attracting endpoint $\xi_i$. Such an element exists by the well-known properties of the group $\SL(2)$  over a non-Archimedean local field,  and its corresponding Bruhat--Tits tree with the associated ideal boundary.

Let $g \in G$. If $g^{-1}(x_0) \in T_H$, there is $h \in H $ such that $hg^{-1}(x_0)\in V$.

If $g^{-1}(x_0) \in T_{G} \setminus T_{H}$, then let $y \in T_H$ be the projection of $g^{-1}(x_0)$ on the tree $T_{H}$; this projection $y$ is unique. Then again there is $h \in H $ such that $h(y) \in V$, and the geodesic $h([y,g^{-1}(x_0)])$ is disjoint from $T_{H}$, except the point $h(y)$. By left multiplying by  an element $h'$ in the $H$-stabilizer of $h(y) \in V \subset T_{H}$, we can suppose that $h'h([y,g^{-1}(x_0)]) \subset [x_i,\xi_i)$ for some $i \in I_H$. As we acted with type-preserving elements (i.e. $h,h',g$ are all type-preserving), there is $n\geq 1$ such that $a_i^{-n}h'hg^{-1}(x_0) \in V$. 

Let \editA{$\mathcal{K}$ be set of all elements in $G$ that  send the vertex $x_0$ to one of vertices of $V\subset T_G$}.  Notice $\mathcal{K}$ is a compact subset of $G$. Then for both cases $hg^{-1}(x_0) \in V$, \edit{resp.  $a_i^{-n}h'hg^{-1}(x_0) \in V$}, we have that  $hg^{-1} \in \mathcal{K}$,  \edit{ resp. $a_i^{-n}h'hg^{-1} \in \mathcal{K}$.} This implies that $g^{-1} \in H\mathcal{B}\mathcal{K}$ and thus $ g \in \mathcal{K}\mathcal{B}H$ as required. 
\end{proof}

\begin{remark}  Notice Theorem \ref{prop::polar_decom} gives $\mathcal{B}$ as a union of a possibly infinite number of \edit{$A_i$'s} that are pairwise non  $H$-conjugate. Following results of \cite{Hel_k_invol, HelWD}, the polar decomposition $\mathcal{K}\mathcal{B}H$ of \cite{BenoistOh} for $k$-involutions  has a finite number of such \edit{$A_i$'s} in the union $\mathcal{B}$.  However, in the next section we use Lemmas \ref{lem::nr_orbits_SL(F)_SL(E)}, \ref{lem::nr_orbits_diag_SL}, \ref{lem::nr_orbits_H_SL} for the pairs $(G,H) \in \{(\SL(2,F),H_{\theta_a}), (\SL(2,E),\SL(2,F))\}$ which give us a finite number of \edit{$A_i$'s} in our decomposition for $\mathcal{B}$. 
\end{remark}

\section{Chabauty Limits of $\SL(2, F)$ inside $\SL(2, E)$ for quadratic $E/F$}\label{sec:chab_quad}%%%%%%%%%%%%%%%%%%%%%%%%%%%%%

We will use the notation and conventions from Section \ref{sec:background}. Let $F$ be a finite field-extension of $\QQ_p$ and $E$ be any quadratic extension of $F$. \editA{The groups $B^\pm$ will be used in this section and the next. }

%Let $F$ be a finite field-extension of $\QQ_p$ and $E$ be any quadratic extension of $F$. Let $k_F, k_E$ be the residue fields of $F,E$, respectively,  and $\omega_F, \omega_E$ be  uniformizers of $F,E$, respectively. Recall $k_F^{*}/(k_F^{*})^{2} =\{1, S\}$, for some non-square $S \in k_F^{*}$. Then $F^{*}/(F^{*})^{2} =\{1, \omega_F, S, S\omega_F\}$, and $E$ can only be one of the fields $F(\sqrt{\omega_F})$, $F(\sqrt{S \omega_F})$ (the ramified extensions and $\omega_F \neq \omega_E$), or $F(\sqrt{S})$ (the unramified extension and $\omega_F = \omega_E$). For a valuation on $E$ we choose the unique valuation $\vert \cdot \vert_E$ that extends the given valuation $\vert \cdot \vert_F$ on $F$. Choose $\alpha \in \{\sqrt{\omega_F}, \sqrt{S}, \sqrt{S\omega_F}\}$ and so $E=F(\alpha)$. Notice each element $x \in E$ can be uniquely written as $x = a+b \alpha$, with $a,b \in F$. For the ramified extensions we can consider $\omega_E^2 = \omega_F$.

The stabilizer in  $\SL(2, F)$, resp. $\SL(2, E)$, of the endpoint  $\big[ \begin{smallmatrix} 1 \\0 \end{smallmatrix}\big] $ is the Borel subgroup 
\edit{ $$B_F^{+}:=\left \{ \left( \begin{smallmatrix} a &b \\ 0 &a^{-1}\end{smallmatrix} \right) \;|\; b\in F, a \in F^{\times} \right \}, \qquad \text{resp. } B_E^{+}:=\left \{ \left( \begin{smallmatrix} x &y \\ 0 &x^{-1}\end{smallmatrix} \right) \;|\;  y\in E, x \in E^{\times} \right\}. $$}
The stabilizer in  $\SL(2,F)$, resp. $\SL(2, E)$, of the endpoint  $\big[ \begin{smallmatrix} 0 \\1 \end{smallmatrix}\big] $  is the opposite Borel subgroup 
\edit{ $$B_F^{-}:=\left \{ \left( \begin{smallmatrix} a &0 \\  b&a^{-1}\end{smallmatrix} \right)\;|\;  b\in F, a \in F^{\times} \right \} , \qquad \text{resp. }  B_E^{-}:= \left \{ \left( \begin{smallmatrix} x &0\\ y &x^{-1}\end{smallmatrix} \right) \;|\;  y\in E, x \in E^{\times} \right\}. $$}

Recall from Lemma \ref{lem::nr_orbits_SL(F)_SL(E)} there are at most $5$ $\SL(2,F)$-orbits on the boundary $\partial T_{E}\cong P^1E=E \cup \{\infty\}$: the $\SL(2,F)$-orbit of  $\big[ \begin{smallmatrix} 1 \\0 \end{smallmatrix}\big] $, and the $\SL(2,F)$-orbits of $\big[ \begin{smallmatrix} 1 \\m \end{smallmatrix}\big] $, for each $m \in F^{*}/(F^{*})^{2}$,  that might coincide for two different $m$'s. By the polar decomposition from Proposition \ref{prop::polar_decom} applied to the pair $(G,H)=(\SL(2,E), \SL(2,F))$   and from Lemma \ref{lem::nr_orbits_SL(F)_SL(E)}, the set $\mathcal{B}$ is a finite union of $A_i$. Since we want to compute the Chabauty limits of $\SL(2,F)$ inside $\SL(2,E)$, \editA{using the polar decomposition $ \mathcal{K}\mathcal{B}H$ and the fact $\mathcal{K   }$-conjugation will only rotate $\SL(2,F)$,} it is enough to compute the Chabauty limits of $\SL(2,F)$ under conjugation by a sequence of elements from some fixed $A_i \subset \mathcal{B}$. Because we want to choose a group $A_i \subset \mathcal{B}$ such that the corresponding computations will be easier,  we rotate $\SL(2,F)$ in such a way that the chosen $A_i$ is generated by the diagonal matrix \edit{ $\left \{\left( \begin{smallmatrix} w_{E}& 0 \\  0 & w_E^{-1} \end{smallmatrix}  \right)\right\}$} of $\SL(2,E)$, which is a hyperbolic element of translation length $2$ along the  bi-infinite geodesic line  $( \big[ \begin{smallmatrix} 0 \\1 \end{smallmatrix}\big] ,   \big[ \begin{smallmatrix} 1 \\0 \end{smallmatrix}\big]  ) \subset T_E$. Such a rotation will  affect the  Chabauty limits of $\SL(2,F)$ only up to $\SL(2,E)$-conjugation.  

We apply this idea and choose the two endpoints  $\big[ \begin{smallmatrix}  1 \\  \alpha \end{smallmatrix} \big], \big[ \begin{smallmatrix} 1 \\ 2 \alpha \end{smallmatrix} \big] \in \partial T_E- \partial T_F$.  Notice 
 \[ \left( \begin{matrix} 1& 0\\ \alpha &1\end{matrix}  \right)  \left[ \begin{matrix} 1 \\0 \end{matrix}\right] = \left[ \begin{matrix} 1 \\\alpha \end{matrix}\right] \qquad \textrm{and} \qquad \left( \begin{matrix} (2\alpha)^{-1}& 1 \\  0 & 2\alpha \end{matrix}  \right) \left[ \begin{matrix} 0 \\1 \end{matrix}\right] = \left[ \begin{matrix} 1 \\2 \alpha \end{matrix}\right].\] 

We conjugate $\SL(2,F)$ by $\left( \begin{smallmatrix} (2\alpha)^{-1}& 1 \\  0 & 2\alpha \end{smallmatrix}  \right) \left( \begin{smallmatrix} 1& 0\\ \alpha &1\end{smallmatrix}  \right)=  \left( \begin{smallmatrix} (2\alpha)^{-1} + \alpha& 1 \\  2\alpha^2 & 2\alpha \end{smallmatrix}  \right)$. We obtain
\begin{equation*}
\begin{split}
&H:=\left( \begin{smallmatrix} (2\alpha)^{-1} + \alpha& 1 \\  2\alpha^2 & 2\alpha \end{smallmatrix}  \right) \SL(2,F) \left( \begin{smallmatrix} 2\alpha& -1 \\ - 2\alpha^2 & \alpha+ (2\alpha)^{-1} \end{smallmatrix}  \right)=\\
&= \edit{\left \{ \left( \begin{smallmatrix} a +2a\alpha^2 -2d \alpha^2 + \alpha(2c-b-2b\alpha^2)&\;\; \left((2\alpha)^{-1} + \alpha\right) (b((2\alpha)^{-1} + \alpha)+d-a)-c\\ 4c\alpha^2 -4b\alpha^4 + \alpha(4 a\alpha^2-4d\alpha^2) & -2a \alpha^2 +2d\alpha^2 +d +\alpha(b+2b\alpha^2-2c) \end{smallmatrix} \right) \;\vert\; a,b,c,d \in F, \; ad-bc=1 \right \}},
\end{split}
\end{equation*}
the subtree $ T_H:=\left( \begin{smallmatrix} (2\alpha)^{-1} + \alpha& 1 \\  2\alpha^2 & 2\alpha \end{smallmatrix}  \right) T_F \subset T_E$ and its ideal boundary $ \partial T_H:= \left( \begin{smallmatrix} (2\alpha)^{-1} + \alpha& 1 \\  2\alpha^2 & 2\alpha \end{smallmatrix}  \right) \partial T_F$ are invariant under $H$. %and its corresponding ideal boundary is $ \partial T_H:= \left( \begin{smallmatrix} (2\alpha)^{-1} + \alpha& 1 \\  2\alpha^2 & 2\alpha \end{smallmatrix}  \right) \partial T_F$. 
It is easy to see that the endpoints  $\big[ \begin{smallmatrix} 0 \\1 \end{smallmatrix}\big] $ and  $\big[ \begin{smallmatrix} 1 \\0 \end{smallmatrix}\big] $  are not in $\partial T_H$, and the bi-infinite geodesic line  $(  \big[ \begin{smallmatrix} 0 \\1 \end{smallmatrix}\big]  ,\big[ \begin{smallmatrix} 1 \\0 \end{smallmatrix}\big]  ) \subset T_E$ intersects the tree $T_H$ either in a vertex, or an edge of $T_E$, depending on whether the extension $E$ is unramified or ramified, respectively.

Let us recall a version of Hensel's Lemma for finite extensions of $\QQ_p$ that will be used below.
\begin{lemma}[Hensel]
 \label{Hensel}
Let $F$ be a finite field-extension of $\QQ_p$. Let $f(X)$ be a polynomial with coefficients in $\mathcal{O}_F$. Suppose there is some $a \in \mathcal{O}_F$ that satisfies:
$$\vert f(a) \vert_F < \vert f'(a)\vert_F^{2}.$$
Then there is a unique $x \in \mathcal{O}_F$ such that $f(x)=0$ in $\mathcal{O}_F$ and $\vert x-a\vert_F < \vert f'(a)\vert_F$. 
\end{lemma}

%Take a sequence of matrices $\{\left( \begin{smallmatrix} a_n& b_n \\  c_n & d_n \end{smallmatrix}  \right)\}_{n\geq 1} \subset  \SL(2,F)$ and the diagonal matrices $\{\left( \begin{smallmatrix} w_E^n& 0 \\  0 & w_E^{-n} \end{smallmatrix}  \right)\}_{n\geq 1} $ that are hyperbolic elements of  $\SL(2,E)$.

\begin{proposition}\label{Chab_limits_SL(2,F)} 
Take a sequence of matrices $\{\left( \begin{smallmatrix} a_n& b_n \\  c_n & d_n \end{smallmatrix}  \right)\}_{n\geq 1} \subset  \SL(2,F)$ and the diagonal matrices \edit{ $\left \{\left( \begin{smallmatrix} \omega_E^n& 0 \\  0 & \omega_E^{-n} \end{smallmatrix}  \right) \right\}_{n\geq 1} $} that are hyperbolic elements of  $\SL(2,E)$. Then 
any limit $ \editA {\left (\begin{smallmatrix}A & B \\ C&D \end{smallmatrix} \right) \in}\SL(2,E)$ of  $$\left( \begin{smallmatrix} a_n +2a_n\alpha^2 -2d_n \alpha^2 + \alpha(2c_n-b_n-2b_n\alpha^2) &\;\; \omega_E^{2n} \left[ \left((2\alpha)^{-1} + \alpha\right) (b_n((2\alpha)^{-1} + \alpha)+d_n-a_n)-c_n \right] \\  \omega_E^{-2n} \left[4c_n\alpha^2 -4b_n\alpha^4 + \alpha(4 a_n\alpha^2-4d_n\alpha^2)\right] & -2a_n \alpha^2 +2d_n\alpha^2 +d_n +\alpha(b_n+2b_n\alpha^2-2c_n) \end{smallmatrix} \right)$$
is of the form $\left( \begin{smallmatrix} a-\alpha b& 0 \\  z &  a+\alpha b \end{smallmatrix} \right)$, with $a,b \in F$ and $a^2 - \alpha^2 b^{2}=1$, $z \in E$. In particular, $z$ can take any value of $E$,  and any solution $a,b \in F$ of the equation $a^2 - \alpha^2 b^{2}=1$ can appear.
\end{proposition}
\begin{proof}
To have a limit in  $\SL(2,E)$ \editA{with respect to the Chabauty topology, by Proposition \ref{prop::chabauty_conv} we need} that in the topology \editA{on $E$} 
\begin{description}
\item[\editA{(A)}]
$\lim\limits_{n\to \infty} a_n +2a_n\alpha^2 -2d_n \alpha^2 + \alpha(2c_n-b_n-2b_n\alpha^2)= A= A_1 + \alpha A_2 \in E$
\item[\editA{(B)}]
$\lim\limits_{n\to \infty}  \omega_E^{2n} \left[ \left((2\alpha)^{-1} + \alpha \right) (b_n((2\alpha)^{-1} + \alpha)+d_n-a_n)-c_n \right]= B= B_1 + \alpha B_2  \in E$
\item[\editA{(C)}]
$\lim\limits_{n\to \infty} \omega_E^{-2n} \left[4c_n \alpha^2 -4b_n \alpha^4 + \alpha(4 a_n \alpha^2-4d_n \alpha^2)\right]= C = C_1 + \alpha C_2 \in E$
\item[\editA{(D)}]
$\lim\limits_{n\to \infty} -2a_n \alpha^2 +2d_n \alpha^2 +d_n +\alpha(b_n+2b_n \alpha^2-2c_n) = D = D_1 + \alpha D_2  \in E$
\end{description}
with all $A_i,B_i,C_i,D_i \in F$.
From \editA{(C)} above \editA{and since $\omega_E^{2} \in F$,} we have that 
$\lim\limits_{n \to \infty} \omega_E^{-2n} \left[4c_n\alpha^2 -4b_n \alpha^4 \right]= C_1$ and 
$\lim\limits_{n\to \infty} \omega_E^{-2n} \left[4 a_n \alpha^2-4d_n \alpha^2 \right]= C_2.$ 
This means that 
 \editA{\begin{equation} \label{equ::8}
\omega_E^{2n} C_{1,n}:= c_n -b_n \alpha^2 \textrm{ and } \omega_E^{2n} C_{2,n}:=a_n - d_n \textrm{ with }\lim\limits_{n\to \infty} C_{i,n}= C_i/(4 \alpha^2) \in F, 
\end{equation}}
%$$\omega_E^{2n} C_{1,n}:= c_n -b_n \alpha^2 \text{ with }\lim\limits_{n\to \infty} C_{1,n}= C_1/(4 \alpha^2) \in F,$$  
%and 
%$$ \omega_E^{2n} C_{2,n}:=a_n - d_n  \text{ with }\lim\limits_{n\to \infty}C_{2,n}= C_2/(4 \alpha^2) \in F$$
implying that 
\begin{equation}\label{zero} \lim\limits_{n\to \infty}\omega_E^{2n} C_{1,n} =  \lim\limits_{n\to \infty}(c_n -b_n \alpha^2)= 0 \textrm { and } \lim\limits_{n\to \infty}\omega_E^{2n} C_{2,n} =  \lim\limits_{n\to \infty}(a_n -d_n)= 0. \end{equation}

Moreover, adding \editA{(A) and (D)}, \editA{and then by (\ref{equ::8})} we have
\begin{equation*}
\begin{split}
\lim\limits_{n\to \infty} (a_n +2a_n \alpha^2 -2d_n \alpha^2  -2a_n \alpha^2 +2d_n \alpha^2 +d_n)&= \lim\limits_{n\to \infty} (a_n  +d_n)\\
&=  \lim\limits_{n\to \infty} (d_n +\omega_E^{2n} C_{2,n}  +d_n)\\
&= \lim\limits_{n\to \infty} 2d_n=A_1 + D_1.\\
\end{split}
\end{equation*}

Thus $a:= \lim\limits_{n\to \infty} a_n=\lim\limits_{n\to \infty} d_n =(A_1 + D_1)/2 \in F$.  From the first terms of \editA{(A) and (D)}, we  see $A_1=D_1.$

Adding the \editA{$\alpha$-terms} from \editA{(A) and (D)}, we get $$\lim\limits_{n\to \infty} (2c_n-b_n-2b_n \alpha^2  +b_n+2b_n \alpha^2-2c_n)= 0 =A_2 + D_2 \Rightarrow A_2=-D_2. $$
\editA{From $\lim\limits_{n\to \infty}(c_n -b_n \alpha^2)= 0$ in \eqref{zero}} and from the $\alpha$-term of \editA{(D)} we get that 
$$\lim\limits_{n\to \infty} b_n = D_2 \in F.$$
Thus $b:= \lim\limits_{n\to \infty} b_n =D_2 \in F$ and $\lim\limits_{n\to \infty} c_n= \alpha^2D_2= \alpha^2b\in F$.

\editA{Taking} $a_n d_n -c_n b_n=1$ and replacing $a_n = d_n +\omega_E^{2n} C_{2,n}$ and $c_n = b_n \alpha^2+\omega_E^{2n} C_{1,n}$, we get
$$d_n^2 + \omega_E^{2n} C_{2,n} d_n - \alpha^2 b_n^2 - \omega_E^{2n} C_{1,n}b_n=1 \text{ implying }$$
$$ 1=\lim\limits_{n\to \infty} d_n^2 + \omega_E^{2n} C_{2,n} d_n - \alpha^2 b_n^2 - \omega_E^{2n} C_{1,n}b_n= a^2 - \alpha^2 b^2.$$

\editA{By an easy computation using $ 1\neq \alpha^{2} \in F^{*}/(F^{*})^{2}$ and properties of the norm $\vert \cdot \vert_F$, recall} that $a^2 - \alpha^2 b^2=1$ with $a,b \in F$ implies that $a,b \in \mathcal{O}_F$ with $\vert a \vert_F=1$ and $\vert b \vert_F <1$, \editA{thus $b \in \omega_F  \mathcal{O}_F$.}

In fact, \editA{(A), (B) and (D)} above will become 
\begin{description}
\item[(i)]
$\lim\limits_{n\to \infty} a_n +2a_n \alpha^2 -2d_n \alpha^2 + \alpha(2c_n-b_n-2b_n \alpha^2)= \lim\limits_{n\to \infty} a_n - \alpha b_n= a- \alpha b= A$
\item[(ii)]
$\lim\limits_{n\to \infty}  \omega_E^{2n} \left[ \left((2 \alpha)^{-1} + \alpha \right) (b_n((2 \alpha)^{-1} + \alpha)+d_n-a_n)-c_n \right]=0 = B$
\item[(iii)]
$\lim\limits_{n\to \infty} -2a_n \alpha^2 +2d_n \alpha^2 +d_n +\alpha(b_n+2b_n \alpha^2-2c_n) = \lim\limits_{n\to \infty} d_n +\alpha b_n  =a+ \alpha b= D $.
\end{description}

We will use Hensel's Lemma \ref{Hensel} to show any element $C \in E$ can be obtained in \editA{(C)}. %Indeed, we use Hensel's Lemma \ref{Hensel}. 
Take any $C = C_1 + \alpha C_2 \in E$, and any $b \in \mathcal{O}_F$, such that $a^2 - \alpha^2 b^{2}=1$. Then for $n \geq 1$ large enough, \editA{since $\omega_E^{2} \in  \mathcal{O}_F \subset F$ and $b \in \omega_F  \mathcal{O}_F$,} we have that $ \omega_E^{2n}C_2, \alpha^2 b^2+\omega_E^{2n}C_1 b \in \omega_\editA{F} \mathcal{O}_F$. Take 
$$ f_n(X):= X^2 +\omega_E^{2n}C_2 X -  \alpha^2 b^2-\omega_E^{2n}C_1 b -1, \edit{\text{ then }}  f'_n(X) = 2X + \omega_E^{2n}C_2.$$
 Then $f_n(X) \in \mathcal{O}_F[X]$, 
 \edit{ $$f_n(1)= 1+\omega_E^{2n}C_2  - \alpha^2 b^2-\omega_E^{2n}C_1 b -1 \equiv 0 (mod \; \omega_\editA{F}), $$ 
 $$ \textrm{and } f'_n(1)=2+ \omega_E^{2n}C_2 \equiv 2 (mod \; \omega_\editA{F}).$$}
Then by Hensel's Lemma \ref{Hensel}, there is $d_n \in  \mathcal{O}_F$ such that 
\edit{$$ f_n(d_n):= d_{n}^{2} +\omega_E^{2n}C_2 d_n - \alpha^2 b^2-\omega_E^{2n}C_1 b-1=0, \textrm{ and }d_n \equiv 1  (mod \; \omega_\editA{F}).$$}
Then take $a_n = d_n +\omega_E^{2n} C_{2}$ and $c_n = b\alpha^2+\omega_E^{2n} C_{1}$, getting $a_n d_n -c_n b=1$. Thus \editA{(C)} becomes 
\edit{$$\lim\limits_{n\to \infty} \omega_E^{-2n} \left[4c_n\alpha^2 -4b\alpha^4 + \alpha(4 a_n\alpha^2-4d_n\alpha^2)\right]= 4\alpha^2 C = 4\alpha^2 (C_1 + \alpha C_2) \in E.$$}
Notice that up to extracting a subsequence and using the fact that $d_n \in \mathcal{O}_F$, we have that $\lim\limits_{n \to \infty}a_n=a$, and  $a^2 - \alpha^2 b^{2}=1$. 
\end{proof}

\edit{If we conjugate $B^{-}$ with the matrix $\left( \begin{smallmatrix} 0 &-1 \\ 1 &0\end{smallmatrix} \right)$ we get $B^{+}$, and so} we have just proved the following:
 \begin{theorem}
\label{cor::cartan_lim_SL}
Let $F$ be a finite field-extension of $\QQ_p$ and $E=F(\alpha)$ be any quadratic extension of $F$. Then  any Chabauty limit of $\SL(2,F)$ in  $\SL(2,E)$ is either an $\SL(2,E)$-conjugate of $\SL(2,F)$, or an $\SL(2,E)$-conjugate of the subgroup 
$\{\left( \begin{smallmatrix} a-\alpha b& z \\  0 &  a+\alpha b \end{smallmatrix} \right) \; \vert \;  a,b \in F \text{ with }a^2 - \alpha^2 b^{2}=1,  z \in E \} \leq B_E^{+}.$ 
\end{theorem}

\begin{remark}  Notice this proof does not depend on whether the extension is ramified or unramified. 
\end{remark}

\section{Chabauty Limits of $\SL(2, \RR)$ inside $\SL(2, \CC)$} \label{sec:chab_real}%%%%%%%%%%%%%%%%%%%%%%%%%%%%%

Recall that $\SL(2,\RR)$ is the isometry group of the real hyperbolic plane $\HH^2$ and $\SL(2,\CC)$ is the isometry group of \edit{the} real hyperbolic 3-space $\HH^3$.  Since $\CC$ is a quadratic extension of $\RR$, the situation  mirrors Section \ref{sec:chab_quad}.   The boundary of $\HH^3$ in the Poincar\'{e} ball model is the Riemann sphere, which may be thought of as \edit{the} union of points $\{ \big[ \begin{smallmatrix} 1 \\ x  \end{smallmatrix} \big] | x \in \CC \} \cup \{\big[ \begin{smallmatrix} 0 \\ 1 \end{smallmatrix} \big] \}$.   Since $\SL(2,\RR)$ acts on $\HH^3$ and its boundary, one may easily compute  that the stabilizer in $\SL(2,\RR)$ of the endpoint  $\big[ \begin{smallmatrix} 1 \\  i \end{smallmatrix} \big]$ is the compact subgroup $\{  \left( \begin{smallmatrix} a &b \\ - b &a\end{smallmatrix} \right)\; | \; a^2 +  b^2 =1 \; a,b\in \RR \} $.

For $k= \RR$ or $\CC$, the following stabilizers in  $\SL(2, k)$ are the Borel subgroups:
 \[ \Stab_{\SL(2,k)}\left( \begin{bmatrix} 1 \\0 \end{bmatrix}\right) = \{ \left( \begin{smallmatrix} a &b \\ 0 &a^{-1}\end{smallmatrix} \right)| b\in k, a \in k^{\times} \}= B^{+}_{k}\]
\[  \Stab_{\SL(2,k)}\left( \begin{bmatrix} 0 \\1 \end{bmatrix}\right) = \{ \left( \begin{smallmatrix} a &0 \\  b&a^{-1}\end{smallmatrix} \right)| b\in k, a \in k^{\times} \} = B^{-}_{k} .\]

Notice that with respect to the complex conjugation on $\CC$, which extends to an involution on $\SL(2,\CC)$, the group $\SL(2,\RR)$  is the fixed point group under complex conjugation. We apply  the polar decomposition (see for example Proposition 7.1.3 of \cite{HS}) to the pair $(\SL(2,\CC), \SL(2,\RR))$. More precisely, let $G$ be a semisimple Lie group with finite center and $H$ a symmetric subgroup. \editA{Let $\mathfrak g = \mathfrak h \oplus \mathfrak q$ be the decomposition into positive and negative eigenspaces, and $\mathfrak g = \mathfrak k \oplus \mathfrak p$ be the Cartan decomposition.} Let $\mathcal{K}$ be a maximal compact subgroup of $G$ and $\mathcal{B}$ the exponential of \editA{$\mathfrak b= \mathfrak a \cap \mathfrak q $, which is} a subalgebra of the maximal abelian subalgebra $\mathfrak a$ that has a \editA{non-empty} intersection with the positive and negative eigenspaces.  
 Then: 

\begin{proposition}[Proposition 7.1.3 of \cite{HS}] For any $g \in G$ there exists $k \in \mathcal{K}, b \in \mathcal{B}, h \in H$ such that $g = kbh$. Moreover $b$ is unique up to conjugation by the Weyl group $W_{H\cap \mathcal{K}}$. 
\end{proposition} 

Alternatively, for the pair  $(\SL(2,\CC), \SL(2,\RR))$ one can just notice that the $\SL(2,\RR)$-orbits on the boundary of  $\HH^3$  are  the $\SL(2,\RR)$-orbits for  $\big[ \begin{smallmatrix} 1 \\0 \end{smallmatrix} \big]$, and  $\big[\begin{smallmatrix} 1 \\i \end{smallmatrix}\big]$, respectively. Then proceed as in the proof of Proposition \ref{prop::polar_decom}.

As a consequence of the polar decomposition above, it suffices to consider Chabauty limits of $\SL(2,\RR)$ inside $\SL(2,\CC)$ by conjugating only by elements in $\mathcal{B}$. Since we want to use the diagonal matrices in order to conjugate and compute the Chabauty limits, we have to rotate/conjugate the subgroup $\SL(2,\RR)$ such that the endpoints  $\big[ \begin{smallmatrix} 1 \\0 \end{smallmatrix} \big]$ and  $\big[ \begin{smallmatrix} 0 \\1 \end{smallmatrix} \big]$ will be transverse to the $\HH^2$-like slice corresponding to the rotated $\SL(2,\RR)$. \editA{Explicitly, we consider the following} action on the endpoints 
\[ \left( \begin{matrix} 1& 0\\ i &1\end{matrix}  \right)\begin{bmatrix} 1 \\0 \end{bmatrix}= \begin{bmatrix} 1 \\  i \end{bmatrix} \qquad 
\left( \begin{matrix} (2i)^{-1}& 1 \\  0 & 2i \end{matrix}  \right)\begin{bmatrix} 0 \\1 \end{bmatrix}= \begin{bmatrix} 1 \\  2i \end{bmatrix}. \]

Thus we conjugate the group $\SL(2,\RR)$ by the matrix $\left( \begin{smallmatrix} (2i)^{-1}& 1 \\  0 & 2i \end{smallmatrix}  \right) \left( \begin{smallmatrix} 1& 0\\ i &1\end{smallmatrix}  \right)=  \editA{\left( \begin{smallmatrix} \frac{i}{2}& 1 \\  -2 & 2i \end{smallmatrix}  \right)}$. We obtain
\begin{equation*}
\begin{split}
&\editA{\left( \begin{smallmatrix} \frac{i}{2}& 1 \\  -2 & 2i \end{smallmatrix}  \right)} \SL(2,\RR) \editA{\left( \begin{smallmatrix} 2i& -1 \\ 2 & \frac{i}{2}\end{smallmatrix}  \right)}=\\
&= \edit{ \left \{ \left( \begin{smallmatrix} -a +2d  + i(2c+b)&\;\; -c -\frac{b}{4} + \frac{i}{2}(d-a)\\ -4c -4b + i(-4 a+4d) & 2a -d  +i(-b-2c) \end{smallmatrix} \right) \vert a,b,c,d \in \RR, \; ad-bc=1\right \}} . 
\end{split}
\end{equation*}

\begin{proposition}
Take a sequence of matrices $\{\left( \begin{smallmatrix} a_n& b_n \\  c_n & d_n \end{smallmatrix}  \right)\}_{n\geq 1} \subset  \SL(2,\RR)$ and the diagonal matrices $\{\left( \begin{smallmatrix} e^n& 0 \\  0 & e^{-n} \end{smallmatrix}  \right)\}_{n\geq 1} $ that are hyperbolic elements of  $\SL(2,\CC)$. Then 
any limit in $\SL(2,\CC)$ of  $$\left( \begin{smallmatrix} -a_n +2d_n  + i(2c_n+b_n) &\;\; e^{2n} \left[  -c_n -\frac{b_n}{4} + \frac{i}{2}(d_n-a_n) \right] \\  e^{-2n} \left[-4c_n -4b_n + i(-4 a_n+4d_n)\right] & 2a_n  -d_n  +i(-b_n-2c_n) \end{smallmatrix} \right)$$
is of the form $\left( \begin{smallmatrix} a-ib& z \\  0 &  a+ib \end{smallmatrix} \right)$, with $a,b \in \RR$ with $a^2 + b^{2}=1$, and $z \in \CC$. In particular, $z$ can be any complex number,  and any solution $a,b \in \RR$ of the equation $a^2 + b^{2}=1$ may appear.
\end{proposition}
\begin{proof}
The proof follows the same computations  as in  Proposition \ref{Chab_limits_SL(2,F)}. Just to summarize the computations one can easily obtain that: $\lim\limits_{n\to \infty} d_n-a_n=0,\lim\limits_{n\to \infty} 2d_n-a_n $ converge in $\RR$,  and $\lim\limits_{n\to \infty} 4c_n + b_n=0$, $\lim\limits_{n\to \infty} 2c_n + b_n$ converge in  $\RR$. Combining those facts, we have that $\lim\limits_{n\to \infty} a_n = -\lim\limits_{n\to \infty} d_n$, $\lim\limits_{n\to \infty} c_n = -1/4\lim\limits_{n\to \infty} b_n$,  all in $\RR$.  This implies 
\[ \lim\limits_{n\to \infty} e^{-2n} \left[-4c_n -4b_n + i(-4 a_n+4d_n)\right] =0 \]  
\[  \lim\limits_{n\to \infty}  -a_n +2d_n  + i(2c_n+b_n) = a-ib, \textrm{ and } \lim\limits_{n\to \infty} 2a_n  -d_n  +i(-b_n-2c_n) = a+ib,\] 
with $a^2+b^2=1$ and $a,b \in \RR$.

To prove we can obtain any matrix of the form  $\left( \begin{smallmatrix} a-ib& z \\  0 &  a+ib \end{smallmatrix} \right)$, take any $b$ and $z=z_1+iz_2 \in \CC$ with the required properties. Then:
\begin{enumerate}
\item
if $b^2=1$ then we have $\lim\limits_{n\to \infty} a_n = -\lim\limits_{n\to \infty} d_n=0$. Then take $d_n:=z_2e^{-2n}+a_n, b_n:= -4(e^{-2n}z_1 +c_n)$
and solve the equation 
 \edit{$$1= a_n d_n-c_nb_n= a_n^2 + 4c_n^2 + c_n 4z_1 e^{-2n} + a_nz_2e^{-2n}$$}
  for $c_n$. As $a_n$ and $e^{-2n}$ will converge to zero, \editA{for $n$ sufficiently large,} there are real solutions for $c_n$.
\item
if $b^2 \neq 1$ take $c_n:=b/2, d_n:=z_2e^{-2n}+a_n, b_n:= -4(e^{-2n}z_1 +b/2)$ and solve the equation 
\edit{$$1= a_n d_n-c_nb_n=a_n^2 + a_nz_2e^{-2n} + 2bz_1 e^{-2n} +b^2$$}
for $a_n$. When $e^{-2n}$ is very small, \editA{for $n$ sufficiently large,} there are real solutions for $a_n$.
\end{enumerate}
\end{proof}

Geometrically, we see that we \editA{can choose a hyperbolic element to act on the slice} $\HH^2$ \editA{to push it} to any point in the boundary at infinity of $\HH^3$.  
 To summarize, we have proved the following:
\begin{theorem}
\label{thm::cartan_lim_SL_r}
Any Chabauty limit of $\SL(2,\RR)$ inside  $\SL(2,\CC)$ is $\SL(2,\CC)$-conjugate to either $\SL(2,\RR)$, or to the subgroup 
$\{\left( \begin{smallmatrix} a-i b& z \\  0 &  a+i b \end{smallmatrix} \right) \; \vert \;  a,b \in \RR \text{ with }a^2 + b^{2}=1,  z \in \CC\} \leq B_{\CC}^{+}.$ 
\end{theorem}

\begin{remark} Notice this computation is not covered by \cite{CDW} since they only compute limits inside of $\PGL(n, \RR)$.  However, the result can easily be deduced from their Theorem 4.1, which holds for real and complex symmetric subgroups of semi-simple Lie groups. 
\end{remark}

 \section{Chabauty Limits of Symmetric Subgroups of $\SL(2,F)$}\label{sec:chab_sym}%%%%%%%%%%%%%%%%%%%%%%%%%%%%%

Let $F$ be a finite field-extension of $\QQ_p$ with ring of integers denoted by $\mathcal{O_F}$,  $\omega_F$ a uniformizer, and $k_F$ the residue field of $F$.  Recall $\vert F^{*}/(F^{*})^{2} \vert=4$. Let  $A=\left( \begin{smallmatrix} 0 &1 \\ a &0 \end{smallmatrix} \right)$, with $a\in F^{*}/(F^{*})^{2}= \{ 1, \omega_F, S, \omega_F S\}$, with $S$ a non-square in $k_F^{*}$, and $\theta_a:=\iota_A$ the corresponding $F$-involution of $\SL(2,F)$ (see \cite[Section 1.3]{HW_class}). Take $K_a:=F(\sqrt{a})$ the field extension of $F$ corresponding to $a$; if $a =1$ then $F=K_a$, and if $a\neq 1$ then $K_a$ is a quadratic extension of $F$. Then by Corollary \ref{lem::geom_k_inv_SL_2}  we know that  $A$ fixes pointwise only the two endpoints  $\xi_{\pm} := \big[ \begin{smallmatrix} 1 \\ \pm \sqrt{a} \end{smallmatrix} \big] \in  P^{1}K_a$. Take
 $$H_{\theta_a} :=\{g \in \SL(2,F) \; \vert \; \theta_a(g)=g\},$$
and together with the results of \cite[Section 3]{HW_class} \edit{(see Proposition \ref{prop::prop_shape_inv})} we have:
$$ H_{\theta_a}= \Fix_{\SL(2,K_a)}(\{\xi_{-},\xi_{+}\}) \cap \SL(2,F)= \{  \left( \begin{smallmatrix} x &y \\ ay &x\end{smallmatrix} \right) \in \SL(2,F)\;|\; x^2 - ay^2 =1 \}.$$

%\begin{remark}
%For the case $F=\QQ_p$ and $a=1$, the corresponding subgroup $H_{\theta_a}$ is a maximal $\QQ_p$-split torus in $\SL(2,\QQ_p)$, %more precisely is the diagonal subgroup of $\SL(2,\QQ_p)$ conjugated by the matrix {\blue $C= \left(\begin{smallmatrix} 1/2 &-1 \\ 1/2 &1 \end{smallmatrix} \right)$,} 
%{\blue i.e. $H_{\theta_a}= C Diag(2,\QQ_p) C^{-1}$ where $C= \left(\begin{smallmatrix} 1/2 &-1 \\ 1/2 &1 \end{smallmatrix} \right)$,} (see \cite[Section 3]{HW_class} or  Corollary \ref{lem::geom_k_inv_SL_2}). Then {\blue apply} the results of \cite{CLV} {\blue to} the case $\SL(2,\QQ_p)$ and $H_{\theta_a}$ (where they use very different methods than in this article) and {\blue see} the Chabauty limits of $H_{\theta_a}$ are either $\SL(2,\QQ_p)$-conjugates {\blue of either} $Diag(2,\QQ_p)\leq \SL(2,\QQ_p)$, {\blue or} of the subgroup $\{ \mu _2 \cdot \left( \begin{smallmatrix} 1& x \\ 0 &1  \end{smallmatrix} \right)\; |\; x \in \QQ_p  \}$ of the Borel  $B_{\QQ_p}^{+}\leq \SL(2,\QQ_p)$, where $\mu _2$ is the group of {\blue $2$nd} roots of unity in  $\QQ_p$.
%\end{remark}

%For the general case of $\SL(2,F)$ and the subgroups $H_{\theta_a}$,
To compute the Chabauty limits in the general case of $H_{\theta_a}\leq \SL(2,F)$ we  use the polar decomposition $G=\mathcal{K}\mathcal{B}H$ from Proposition \ref{prop::polar_decom} applied to the pair $(G,H)=(\SL(2,F), H_{\theta_a})$. By  Lemmas \ref{lem::nr_orbits_diag_SL}, \ref{lem::nr_orbits_H_SL}  the corresponding $\mathcal{B}$ is a finite union of groups $\mathcal{A}_i$. As in Section \ref{sec:chab_quad} it is enough to compute the Chabauty limits of $H_{\theta_a}$ under conjugation  by a sequence of elements from some fixed $\mathcal{A} \subset \mathcal{B}$. We  choose a group $\mathcal{A} \subset \mathcal{B}$ such that the corresponding computations will be easier.  Notice that $\Diag(2,F) \leq \SL(2,F)$ has fixed endpoints  $ \big[ \begin{smallmatrix} 0 \\1 \end{smallmatrix} \big] , \big[ \begin{smallmatrix}  1 \\0 \end{smallmatrix} \big]  \in \partial T_F$ (which are different from  $\xi_{\pm} := \big[  \begin{smallmatrix}  1 \\ \pm \sqrt{a} \end{smallmatrix} \big] \in  P^{1}K_a$), and it is  transverse to the bi-infinite geodesic line  $(\big[ \begin{smallmatrix} 1 \\ -\sqrt{a} \end{smallmatrix} \big] , \big[ \begin{smallmatrix}  1 \\  \sqrt{a}\end{smallmatrix} \big]) \subset T_{K_a}$. Thus, up to $H_{\theta_a}$-conjugacy, we can take  $\Diag(2,F) \subset \mathcal{B}$. So, we choose $\mathcal{A}$ to be generated by the diagonal matrix \edit{$\left \{\left( \begin{smallmatrix} w_{F}& 0 \\  0 & w_F^{-1} \end{smallmatrix}  \right)\right\}$} of $\SL(2,F)$, which is a hyperbolic element of translation length $2$ along the  bi-infinite geodesic line $( \big[ \begin{smallmatrix} 0 \\1 \end{smallmatrix} \big] , \big[ \begin{smallmatrix}  1 \\0 \end{smallmatrix} \big]) \subset T_F$. This procedure  affects the  Chabauty limits of $H_{\theta_a}$ only up to $\SL(2,F)$-conjugation. \edit{Our goal is to show:} 

\begin{theorem}
\label{thm::k_inv_chabauty}
Let $F$ be a finite field-extension of $\QQ_p$.  Let  $A=\left( \begin{smallmatrix} 0 &1 \\ a &0 \end{smallmatrix} \right)$, with $a\in F^{*}/(F^{*})^{2}$, and $\theta_a:=\iota_A$ the corresponding $F$-involution of $\SL(2,F)$. Take $H_{\theta_a} :=\{g \in \SL(2,F) \; \vert \; \theta_a(g)=g\}$.
Then any Chabauty limit of $H_{\theta_a}$ is either $\SL(2,F)$-conjugate to $H_{\theta_a}$, or to the subgroup \edit{$\{ \mu \left( \begin{smallmatrix} 1& 0 \\ z &1  \end{smallmatrix} \right)\; |\; z \in F , \mu \in \mu_2 \}$} of the Borel  $B_{F}^{-}\leq \SL(2,F)$, where $\mu _2$ is the group of $2^{nd}$ roots of unity in  $F$.
\end{theorem}

\begin{proof} 
\editA{Notice first that if we conjugate $H_{\theta_a}$ with a sequence of elements from the compact set $\mathcal{K}$, then we get an $\SL(2,F)$-conjugate of $H_{\theta_a}$. Next we compute the rest of the Chabauty limits of $H_{\theta_a}$.}

Thus, consider a sequence  \edit{$\left \{\left( \begin{smallmatrix} x_n & y_n \\ y_n a & x_n \end{smallmatrix} \right)\right \}_{n\geq1} \subset    H_{\theta_a}$} and a sequence \edit{ $\left \{ \left( \begin{smallmatrix} \omega_F^{n} & 0 \\ 0 & \omega_F ^{-n}\end{smallmatrix} \right) \right\}_{n\geq1}\subset \mathcal{A}$} such that
\[ \begin{pmatrix} \omega_F ^{n}& 0 \\ 
0 &  \omega_F ^{-n}
\end{pmatrix} 
\begin{pmatrix} 
x_n & y_n \\ 
y_n a & x_n 
\end{pmatrix} 
\begin{pmatrix}  \omega_F ^{-n} & 0 \\ 
0 &  \omega_F ^{n}
\end{pmatrix}  
= 
\begin{pmatrix} 
x_n &\omega_F ^{2n} y_n \\ 
\omega_F ^{-2n}y_n a  & x_n 
\end{pmatrix} 
\xrightarrow{n \to \infty} 
\begin{pmatrix}  
x & y \\ 
z & t 
\end{pmatrix}  \in \SL(2,F).
\] 
Then we must have that $\lim\limits_{n\to \infty}x_n=x=t \in F$. As well, we have $\lim\limits_{n\to \infty} \omega_F ^{-2n}y_n =z/a \in F$ and $\lim\limits_{n\to \infty}\omega_F ^{2n} y_n =y \in F$. Writing $C_n:=\omega_F ^{-2n}y_n$ we get $y_n=\omega_F ^{2n}C_n$, and because  $\lim\limits_{n\to \infty} C_n=z/a \in F$, one obtains $\lim\limits_{n\to \infty} y_n=0 \in F$, so $\lim\limits_{n\to \infty} \omega_F ^{2n} y_n =0 =y \in F$. Since $ \left( \begin{smallmatrix}  
x & y \\ 
z & t 
\end{smallmatrix} \right) \in \SL(2,F)$ with $x=t$ and $y=0$, we get $x^2 = 1$, implying that $x \in \mu _2$ \edit{where $\mu_2$} denotes the group of $2^{nd}$ roots of unity in  $F$.

So far we have proven that a Chabauty limit of $H_{\theta_a}$ is contained in the subgroup \edit{$\{ \mu  \left( \begin{smallmatrix} 1& 0 \\ z &1  \end{smallmatrix} \right)| z \in F, \mu \in \mu_2  \} $,}  which is condition 2 of Proposition \ref{prop::chabauty_conv} for Chabauty convergence. 

It remains to show condition (1) of Proposition \ref{prop::chabauty_conv}. To show equality, fix some $z \in F$ and take $y_n: = \frac{z}{a} \omega_F^{2n}$, thus $\lim\limits_{n\to \infty} y_n=0$. In particular, \editA{since $z,a$ are both fixed,} for every $n$ large enough $y_n \in \omega_F \mathcal{O}_F$. Then, for each such $n$ we  apply Hensel's Lemma \ref{Hensel} to $f_n(x) = x^2 - ay_n ^2 -1$, where $f_n(1) \equiv 0 \; (\text{mod } \omega_F)$ and $f_n'(1) =2 \not \equiv 0 \;(\text{mod } \omega_F)$.  So there is a solution $x_n \in \mathcal{O}_F$ with $f_n(x_n)=0$ and $x_n \equiv 1  \; (\text{mod } \omega_F)$. In particular,  we  have $ x_n^2 - ay_n ^2=1$, $\lim\limits_{n\to \infty} \omega_F ^{-2n}y_n a=z \in F$, $\lim\limits_{n\to \infty} \omega_F ^{2n}y_n =0 \in F$, and $\lim\limits_{n\to \infty} x_n = x$, and $x=1$. By taking $f_n(\omega_F-1)$ and $f_n'(\omega_F-1)$, one obtains $x=-1$. \end{proof}

\begin{remark}
In the case $F=\QQ_p$ and $a=1$, we have $H_{\theta_a}= C \Diag(2,\QQ_p) C^{-1}$ where $C= \left(\begin{smallmatrix} \frac{1}{2} &-1 \\ \frac{1}{2} &1 \end{smallmatrix} \right)$. So $H_{\theta_a}$ is a maximal $\QQ_p$-split torus in $\SL(2,\QQ_p)$, (see \cite[Section 3]{HW_class} or  Corollary \ref{lem::geom_k_inv_SL_2}). Then results of \cite[\editA{Section 8.1}]{CLV} (using different methods than in this article) give the same result as Theorem \ref{thm::k_inv_chabauty}.
\end{remark}

\end{document}